\documentclass[reqno,tbtags,oneside]{amsart}

\usepackage{amssymb,enumitem,cancel}

\usepackage[foot]{amsaddr}
 
\usepackage{mathtools}
\usepackage{mathrsfs}
\usepackage{xfrac}
\usepackage{marginnote}
\usepackage{dsfont}
\usepackage{changepage}
\usepackage{xcolor}
\usepackage{soul}
\usepackage{extdash}
\usepackage{hyperref}
\usepackage{cleveref}

\numberwithin{equation}{section}

\newtheorem{thm}{Theorem}[section]
\newtheorem{prop}[thm]{Proposition}
\newtheorem{lem}[thm]{Lemma}
\newtheorem{cor}[thm]{Corollary}
\newtheorem*{ithm}{Theorem}
\crefname{thm}{Theorem}{Theorems}
\crefname{prop}{Proposition}{Propositions}
\crefname{lem}{Lemma}{Lemmas}
\crefname{cor}{Corollary}{Corollaries}

\crefname{claim}{Claim}{Claims}
\crefname{step}{Step}{Steps}

\theoremstyle{definition}
\newtheorem{defn}[thm]{Definition}
\newtheorem{ex}[thm]{Example}
\crefname{defn}{Definition}{Definitions}
\crefname{ex}{Example}{Examples}

\theoremstyle{remark}
\newtheorem{rmk}[thm]{Remark}
\crefname{rmk}{Remark}{Remarks}

\AddToHook{env/prop/begin}{\crefalias{thm}{prop}}
\AddToHook{env/lem/begin}{\crefalias{thm}{lem}}
\AddToHook{env/cor/begin}{\crefalias{thm}{cor}}
\AddToHook{env/claim/begin}{\crefalias{thm}{claim}}
\AddToHook{env/step/begin}{\crefalias{thm}{step}}
\AddToHook{env/defn/begin}{\crefalias{thm}{defn}}
\AddToHook{env/ex/begin}{\crefalias{thm}{ex}}
\AddToHook{env/rmk/begin}{\crefalias{thm}{rmk}}

\newcommand{\pair}[2]{\ensuremath\langle#1,#2\rangle}
\newcommand{\norm}[1]{\ensuremath\lVert#1\rVert}

\newcommand{\defeq}{\vcentcolon=}
\newcommand{\eqdef}{=\vcentcolon}
\newcommand{\trinorm}[1]{\ensuremath|\!|\!|#1|\!|\!|}

\let\temp\phi
\let\phi\varphi
\let\varphi\temp
\let\temp\varepsilon
\let\varepsilon\epsilon
\let\epsilon\temp

\DeclareMathOperator{\tr}{tr}

\newcommand{\N}{\mathbb{N}}
\newcommand{\Pc}{\mathcal{P}}
\newcommand{\Z}{\mathbb{Z}}
\newcommand{\de}{\partial}
\newcommand{\di}{\mathrm{d}}
\newcommand{\R}{{\mathbb{R}}}

\newcommand{\call}[1]{\ensuremath\mathcal{#1}}
\newcommand{\frk}[1]{\ensuremath\mathfrak{#1}}

\newcommand{\bb}[1]{\ensuremath\mathbb{#1}}

\newcommand\Sym{\mathscr{S}}
\newcommand\trn\intercal

\newcommand\loc{\mathrm{loc}}

\newcommand{\var}{\mspace{1.8mu}\cdot\mspace{1.8mu}}
\newcommand\DSet[1]{\ensuremath[\![#1]\!]}

\newcommand\mres\llcorner

\setlist[itemize]{left=1em, itemsep=1.5pt}
\setlist[enumerate]{left=1em, itemsep=1.5pt}

\begin{document}

\date{\today}

\title[Local existence and uniqueness for infinite-dim.~Nash systems]{Short-time existence and uniqueness for some infinite-dimensional Nash systems}
\author[D.\ F.\ Redaelli]{Davide Francesco Redaelli}
\address{Dipartimento di Matematica \\ Università di Roma Tor Vergata \\ Via della Ricerca Scientifica 1\\ 00133 Roma, Italy}
\email{redaelli@mat.uniroma2.it}

\begin{abstract}
We prove local (in time) existence and uniqueness for a class of infinite-dimensional Nash systems, namely systems of infinitely many Hamilton--Jacobi--Bellman equations set in an infinite-dimensional Euclidean space. Such systems have been recently showed to arise in the theory of stochastic differential games with interactions governed by sparse graphs by Cirant and the author [{\em Dyn.\ Games Appl.}, 15 (2025), pp.\ 558--591], under structural assumptions that inspired the hypotheses exploited in the present work. Contextually, we also prove a general linear result, providing a priori estimates, stable with respect to the dimension, for transport-diffusion equations whose drifts (and their derivatives) enjoy appropriate decay properties.
\end{abstract}

\maketitle

%\tableofcontents

\section{Introduction}

We consider the following system of infinitely many backward parabolic differential equations of Hamilton--Jacobi--Bellman (HJB) type, which we refer to as an \emph{infinite-dimensional Nash system}:
\begin{equation} \label{tdli_ns}
\begin{dcases} -\de_t u^i - \sum_{jk} A^{jk}(t,x) D^2_{jk} u^i + H^i(t,x, \call Du) + \sum_{j\neq i} \de_{p^j} H^j(t,x,\call Du) D_ju^i = 0 \\
u^i(T,\cdot) = G^i\,, \qquad i \in \N\,,
\end{dcases} 
\end{equation}
where $\call Du \defeq (D_ju^j)_{j\in\N}$, set in $\R^\omega_T \defeq [0,T] \times \R^\omega$, with $\R^\omega$ being the space $\R^\N$ equipped with the product topology.

The unknowns are the \emph{value functions} $u^i \colon \R^\omega_T \to \R$; the data are the \emph{horizon} $T$, the \emph{diffusion} $A = (A^{jk})_{j,k \in \N} \colon [0,T] \to \R^{\omega \times \omega}$, the \emph{Hamiltonians} $H^i \colon \R^\omega_T \times \R^{\omega} \to \R$ and the \emph{terminal costs} $G^i \colon \R^\omega \to \R$.

The notation $\R^{\omega \times \omega}$ indicates infinite-dimensional square matrices, or equivalently the space $(\R^\omega)^\omega$, in analogy with the notation $\R^{N \times N}$ for $N \times N$ matrices. The typical element of $[0,T] \times \R^\omega \times \R^{\omega}$ is denoted by $(t,x,p)$, with coordinates $x^j$ and $p^j$, $j \in \N$. Also, we will use the notation $\R^\omega_x$ and $\R^\omega_p$ when we need to specify which copy of $\R^\omega$ we are considering, while $\R^\omega_T$ will always mean $[0,T] \times \R^\omega_x$.

Derivatives are understood in the sense of Gateaux. Recall that given $k$ vectors $v^1,\dots,v^k \in \R^\omega$, the $k$-th Gateaux derivative of $V$ along them is recursively defined by
\[
D^kV(x;v^1,\dots,v^k) \defeq \frac{\de}{\de s}\bigg|_{0} D^{k-1}V(x+sv^k;v^1,\dots,v^{k-1})\,.
\]
We denote by $\de_t$ the derivative with respect to $t$, by $D_j$ the derivative with respect $x^j$ -- that is, with respect to $x$ along $e_j$ -- and by $\de_{p^{j}}$ the derivative with respect to $p^{j}$.

\smallskip

Finite-dimensional systems like \eqref{tdli_ns} -- that is, with $i \in \DSet{N} \defeq \{0,\dots,N-1\}$ and thus set in $\R^N_T \defeq [0,T] \times \R^N$ -- naturally arise to describe closed-loop Nash equilibria in stochastic differential $N$-player games, and they are notoriously hard to study due to the strong coupling Hamiltonian terms. Let us briefly recall how this happens, referring the reader to \cite{cardelar,Friedman72} for further details.

Consider a game in which the vector $X_t = (X^i_t)_{i \in \DSet N}$ of the states of the players obeys the following $\R^N$-valued SDE on $[0,T]$:
\[
\di X_t = B(t,X_t,\alpha_t)\,\di t + \Sigma(t,X_t) \,\di W_t\,,
\]
where $B = (B^i)_{i \in \DSet N} \colon [0,T] \times \R^N \times \R^N \to \R^N$ is the given vector of the \emph{drifts}, $\alpha = (\alpha^i)_{i \in \DSet N} \colon [0,T] \to \R^N$ is the vector of closed-loop controls in feedback form (that is, $\alpha_t^i = \varphi^i(t,X_t)$ for some $\varphi^i \colon [0,T] \times \R^N \to \R$), $\Sigma \colon [0,T] \times \R^N \to \R^{N \times M}$ is the \emph{volatility} matrix and $W$ is an $\R^M$-valued Brownian motion. Suppose that each player $i$ aims at minimising a cost of the form
\[
J^i(\alpha) = \bb E \biggl[\, \int_0^T F^i(t,X_t,\alpha_t) \,\di t + G^i(X_T) \biggl]
\]
and that there exists a unique $\alpha^{*}(t,x,q)$ (with $q \in (\R^N)^N$) such that, for each $i \in \DSet N$,
\[
\alpha^{*i}(t,x,q) \in \arg\min_{\alpha^i} \Bigl(B\bigl(t,x,(\alpha^i,\alpha^{*,-i}(t,x,q))\bigr) \cdot q^i + F^i\bigl(t,x,(\alpha^i,\alpha^{*,-i}(t,x,q))\bigr)\Bigr)\,.
\]
Then one defines the Hamiltonians
\begin{equation} \label{NST_defHi}
\tilde H^i(t,x,q) \defeq -B(t,x,\alpha^*(t,x,q))) \cdot q^i - F^i(t,x,\alpha^{*}(t,x,q))\,,
\end{equation}
and the optimal controls -- in the sense of Nash equilibrium -- for the game are expected to be given by $\alpha^*(t,x,Du)$, where the so-called value functions $u^i \colon [0,T] \times \R^N \to \R$ solve the Nash system
\begin{equation} \label{NST_preNS}
-\de_t u^i - \tr\bigl(A(t,x) D^2 u^i\bigr) + \tilde H^i(t,x,Du) = 0\,,
\end{equation}
with $A = \frac12\Sigma\Sigma^\trn$ and terminal condition $u^i(T,\var) = G^i$. Now, by the envelope theorem (see, e.g., \cite[Lemma~1.1]{CPnotes}) one has $B(t,x,\alpha^*(t,x,q)) = - \de_{q^i} \tilde H^i(t,x,q)$; therefore, if in addition $B^i$ is independent of $\alpha^{-i}$, then we can consider $\alpha^{*i}$ (and thus also $H^i$) that is independent of $q^{jk}$ for $j \neq k$. In this situation one gets to a system of the form \eqref{tdli_ns} by letting $H^i(t,x,(q^{jj})_{j}) \defeq - B^i(t,x,\alpha^{*i}(t,x,(q^{jj})_{j}))\, q^{ii} - F^i(t,x,\alpha^*(t,x,(q^{jj})_{j}))$.

Note that in this setting we are considering each state $X^i_t$ to be one-dimensional, while in general one could consider $X^i_t \in \R^d$. Our choice is only made for the sake of a simpler notation, as the willing reader can check that our results can be easily extended to the $d$-dimensional case, where one has $(\R^d)^\omega$ instead of $\R^\omega$.

\smallskip

In general, system~\eqref{NST_preNS} is not well-defined for $i \in \N$, and thus in $[0,T] \times \R^\N$; when formally writing its infinite-dimensional analogue, two series appear which are not guaranteed to converge: one coming from the trace and another one carried by the Hamiltonian due to the scalar product in its definition~\eqref{NST_defHi}.

The most notable example that such a system cannot be taken to the limit ``as it is'' is offered by the Mean Field Game (MFG) theory, developed in the first decade of the century independently by Lasry and Lions~\cite{LL07} and Caines, Huang and Malhamé~\cite{hcm} to study Nash equilibria in games with many players satisfying certain symmetry and negligibility assumptions inspired by the mean field paradigm of statistical mechanics. As described in \cite{CDLL} expanding ideas of \cite{LionsSeminar}, in the MFG setting the natural limit of the Nash system is provided by the Master Equation, a PDE of hyperbolic nature set in the space $[0,T] \times \R \times \Pc(\R)$, where $\Pc$ stands for probability measures. In other words, we can say that, exploiting the structural hypotheses of MFG and letting $N \to \infty$ one passes from $\R^N$ to $\R \times \Pc(\R)$ instead of $\R^\N$.

On the other hand, in recent years other settings for large populations games than the MFG one have been studied with rising interest. In particular, in \emph{network games} one drops the general symmetry and negligibility assumptions of MFGs, which basically lead to assume that the functions $B$, $F$ and $G$ above depend on $X$ and $i$ only through $X^i$ and the mean $\frac1{N-1} \sum_{j \neq i} X^j$; instead, it is assumed that interactions are governed by a graph $\Gamma = (V,E)$, with $\#V = N$, in such a way that players are labelled by $v \in V$ and, for example, $G^v(X) = G^v(X^v, \sum_{(v',v) \in E} w(v,v') X^{v'})$ for some weights $w$.

When, as one lets $N \to \infty$, the corresponding sequence of graphs consists of \emph{dense} graphs converging -- in an appropriate sense -- to a limit object which is a \emph{graphon}, it is possible to find conditions that allow to pass to the limit in a spirit analogous to that of MFGs (basically exploiting suitable generalisations of the Law of Large Numbers), thus entering the theory of Graphon Mean Field Games (GMFGs); see, e.g., \cite{Carmona2,BWZ,CainesHuang,LackerSoretLabel} and the references therein.

Relatively few results are available, instead, for \emph{sparse} graphs; the reader can have a look at \cite{CainesIEEE22,CainesHuACC24,CainesHuIEEE24,LackRamWu} and references therein. Studying large population games governed by such graphs is intrinsically more difficult due to the general lack of any overall effect that could come into play as the number of players increases. In (G)MFGs (most of) the other players have an asymptotically negligible impact on a given one, which is though influenced by the mean distribution of them altogether, thus making their average not negligible in the large population limit. One the other hand, if for instance the sequence of graphs has eventually constant maximum degree, then a given player is not affected at all -- at least directly -- by most of the others. What one expects in this case is that mutually distant players -- in the sense of the graph -- should have an unimportant impact on one another, so that a characteristic property of such games should be the possibility for a player to neglect all those too far from it and still play almost optimally. This situation is indeed antipodal to that of (G)MFGs, since a natural way to obtain $\epsilon$-Nash equilibria for games with many players turns out to be cutting some of them out of the game instead of supposing them to be infinitely many in order to exploit the aforementioned typical effects occurring in (G)MFG theory.

\smallskip

The above claims have been recently investigated and proved to be true by the author and M.~Cirant in peculiar Linear-Quadratic (LQ) settings for closed-loop equilibria \cite{CR24}, as well as in more general sparse-graph-based open-loop games \cite[Chapter~1]{PHDT}. In particular, the way of depicting the \emph{unimportance of distant players} proposed in \cite{CR24} consists in showing that, if the graphs have vertices $V = \DSet N$, then $D_j u^i \to 0$ as $N,|i-j| \to \infty$ ($u$ being the solution to the corresponding Nash system), whereas in (G)MFGs one has $D_j u^i \to 0$ for any $j \neq i$. Indeed, note that the derivative of the $i$-th value function along the direction $j$ vanishing is interpreted as the $i$-th player's optimal behaviour disregarding the $j$-th player, and, more in general, a small absolute value of $D_j u^i$ is associated to a small impact of the $j$-th player's behaviour on the $i$-th player's optimal strategy.

As a byproduct, we showed that there exist structural conditions under which $D_j u^i$ vanish fast enough for the series appearing in the formally-written infinite-dimensional Nash system to be convergent. In other words, a consequence of the unimportance of distant players can be the possibility to preserve the formal structure of the Nash system in the limit $N \to \infty$. This is done in \cite{CR24} in a LQ framework, so the purpose of the present work is to make use of the takeaways of that case study in order to formulate suitable assumptions to prove short-time existence and uniqueness for more general infinite-dimensional Nash systems.

More precisely, the strategy in \cite{CR24} is based on the existence of a positive \emph{c-self-controlled} sequence $\beta \in \ell^1(\Z)$ -- the \emph{c} standing for \emph{convolution} -- providing bounds for the derivatives on the data; that is, for instance,
\begin{equation} \label{tdli_cscbound}
|D^2_{hk} G^i| \lesssim \beta^{i-h}\beta^{i-k}\,,
\end{equation}
with $\beta$ satisfying the peculiar property to control its discrete self-convolution:
\[
(\beta \star \beta)^i \defeq \sum_{j \in \Z} \beta^{i-j} \beta^j \lesssim \beta^i \qquad \forall\, i \in \Z\,. 
\]
Thinking of the simpler form of \eqref{tdli_ns} where the equations are given by
\[
-\de_t u^i - \Delta u^i + \frac12|D_i u^i|^2 + \sum_{j \neq i} D_j u^j \cdot D_j u^i = 0\,,
\]
the reader can have a hint that \eqref{tdli_cscbound}, if proved to hold for $u^i$ as well, allows to obtain nice a priori estimates by efficiently handling the strongly coupled terms $\sum_{j \neq i} D_j u^j \cdot D_j u^i$.

\smallskip

Our main result is stated as \Cref{tdli_thmone}; we reproduce it here for the convenience of the reader and of the following discussion.

\begin{ithm} %\label{tdli_thmone}
Let $\beta \in \ell^{\frac12}(\bb Z;\R_+)$ be even and such that $\beta \star \beta \leq c\beta$. For any $i \in \N$, define $\beta_i \defeq \beta^{i-\var} \in \ell^{\frac12}(\Z,\R_+)$. Assume that the following hypotheses are fulfilled, uniformly in the parameters $i,j,k \in \N$:
\begin{itemize}[leftmargin=*]
\item $A \colon \R^\omega_T \to \ell^\infty(\N^2)$ is uniformly positive on $\ell^2(\N)$, with $A^{ij} = A^{ji} \in C^{\frac\gamma2,\gamma}(\R^\omega_T)$ for some $\gamma \in (0,1)$, $D_k A^{ij} = 0$ unless $i=j=k$, and $D_k A^{kk} \in C^{\frac\gamma2,\gamma}(\R^\omega_T) \cap C^0([0,T];C^{2}(\R^\omega))$;
\item $G^i \in C^{3+\gamma-}_{\beta_i}(\R^\omega)$;
\item $\de_{p^i} H^i \in C^0([0,T];C^2_{\beta_i}(\R^\omega_x \times \Omega)) \cap C^0(\R^\omega_T;C^{3-}_{\beta_i}(\Omega))$ for any $\infty$-bounded $\Omega \subset \R^\omega_p$,\footnote{That is, $\sup_\Omega |\var|_\infty \defeq \sup_{p\in\Omega} \sup_i |p^i| < \infty$.} also $D_j \de_{p^i} H^i,\, \de_{p^j} \de_{p^i} H^i \in C^0(\Omega;C^{\frac\gamma2,\gamma}(\R^\omega_T))$.
\end{itemize}

Then there exists $T^*>0$ such that if $T < T^*$ the Nash system~\eqref{tdli_ns} has a unique classical solution $u$, which also belongs to $\ell^\infty(\N;C^0([0,T];C^{2,1}_{\beta_i}(\R^\omega)) \cap C^{\frac12}([0,T];C^{2-}_{\sqrt{\beta_i}}(\R^\omega)))$.
\end{ithm}

The uniform positivity of $A(t,x) \in \ell^\infty(\N^2)$ means that there exists $\lambda \in \R_+$ such that $A(t,x)\xi\cdot\xi \geq \lambda|\xi|^2$ for all $\xi \in \ell^2(\N)$, $(t,x) \in \R^\omega_T$, so it is a uniform ellipticity assumption for the associated differential operator appearing in the Nash system. The spaces of functions appearing in the above statement are defined and discussed in \Cref{NST_secholsp}; let us just mention that they are convenient generalisations of classic H\"{o}lder spaces on the infinite-dimensional domain $\R^\omega$, possibly weighted according to the sequence $\beta$, which is an efficient tool to deal with the diffusion and -- most importantly -- the drift terms, as previously mentioned. For instance, $G^i \in C^{3+\gamma-}_{\beta_i}(\R^\omega)$ implies, in particular, that
\begin{equation} \label{tdli_cscbound2}
|D_{hk}^2 G^i| \lesssim \beta^{i-h} \wedge \beta^{i-k} \wedge \sqrt{\beta^{i-h}\beta^{i-k}}\,,
\end{equation}
which is a weaker form of \eqref{tdli_cscbound}, seemingly more manageable in a non-LQ context. The notion of solution to \eqref{tdli_ns} is specified in \Cref{NST_csoldef}.

The strategy of the proof consists in obtaining such a solution as the limit of solutions to finite-dimensional Nash systems, rather than working directly in an infinite-dimensional setting. 
This ``bottom-up'' approach is reminiscent of that pursued in \cite{CR_ENS}, where a weak formulation of the Master Equation is deduced from a detailed analysis of $N$-dimensional Nash systems in a generalised MFG framework. Herein just as therein, the key step consists in proving certain a priori estimates for the $N$-dimensional system that are stable with respect to $N$, and in this work we will need more precise index-wise a priori estimates in order to have a sufficiently fine control on all partial derivatives of the solutions (because we wish to exploit bounds like that in \eqref{tdli_cscbound2}, provided by the c-self controlled sequence).

The pivotal result we prove to this end is stated as \Cref{tdli_proppret}, which we believe to be of interest in itself as well; it provides, for solutions to linear drift-diffusion equations on an arbitrary -- finite -- time horizon, a priori estimates that are stable with respect to the dimension, provided that the data -- drift included -- fulfil bounds like \eqref{tdli_cscbound2}.

Finally, we point out that the above theorem only providing existence for sufficiently small time horizons (while under our other assumptions uniqueness always holds; see \Cref{NST_thmex}) is to be considered a natural consequence of our assumptions only concerning the regularity and the decay of the data. Indeed, in MFG theory it is well-known that having long-time well-posedness of the Master Equation (which -- we remind -- is the limit object corresponding to our infinite-dimensional Nash system in the symmetric and dense setting) requires additional structural assumptions of \emph{monotonicity}, the most influential ones being the Lasry--Lions monotonicity and the displacement monotonicity; the reader can have a look at \cite{Ahuja16, BM24, CDLL, CiPo, gangbo, GraMe, MouZhang}, and references therein.

Outside the MF context, it is not clear how an analogous assumption should be formulated. A sufficient condition in order to deal with long horizons in our sparse setting is presented in \cite{CR24}, yet it is so tailored to the LQ structure that adapting it to other frameworks seems unviable. Summing up, proving long-time well-posedness for the problem we are considering should require the identification of some correct notion of ``monotonicity'' in this sparse setting, but this is beyond what seems to be achievable at the moment.

\smallskip

We conclude these introductory pages by stressing that, despite the strong relationship between the infinite-dimensional Nash system and network games, this work deals with system~\eqref{tdli_ns} from a purely PDE viewpoint. To the best of our knowledge, our approach is novel among those adopted so far to study infinite-dimensional -- or \emph{abstract} -- HJB equations (we refer, e.g., to \cite{BLL_HJBID,BD_HJBID,HJBHHS,DaPrato_Zabczyk_2002,Masiero16,Masiero14}, and references therein), for multiple reasons. First, as already noted, we mainly work on finite-dimensional problems, deducing estimates that are stable with respect to the dimension. Second, our method does not make use of any tool related to spectral or Hilbert spaces theories, rather it relies on particular decay properties of the data $G$ and $H$. Third, our ambient space $\R^\omega$ is not even Banach, though we note that our strategy could be set in $\ell^\infty(\N)$, which is the most natural Banach space for our Nash system; in that case, considering $A$ to be $\ell^\infty(\N^2)$-valued, we should expect to work with solutions such that $D^2 u^i$ -- in the Fréchet sense -- belongs to $\ell^1(\N^2)$, which is indeed consistent with our decay assumptions.

\subsection*{Acknowledgements}

The author acknowledges the support of the Italian (EU Next Gen) PRIN project 2022W58BJ5 (\emph{PDEs and optimal control methods in mean field games, population dynamics and multi-agent models}, CUP E53D23005910006), of Gruppo Nazionale per l’Analisi Matematica, la Probabilità e le loro Applicazioni (GNAMPA) of the Istituto Nazionale di Alta Matematica (INdAM), and of Fondazione Cassa di Risparmio di Padova e Rovigo.

\section{Weighted H\"{o}lder spaces of functions on \texorpdfstring{$\R^\omega$}{Rw}} \label{NST_secholsp}

We start with some useful definitions and related comments. At first sight, the reader might find the spaces defined in this section a bit odd, nevertheless we point out that they will turn out to be appropriate for dealing with decays governed by a c-self controlled sequence.

\begin{defn}
For $V \colon \R^\omega \to \R$ and $\gamma \in [0,1]$ we define
\[
[V]_{\gamma} \defeq \sup_{j \in \N} \sup_{x^{-j} \in \R^\omega} \sup_{y^j \neq z^j} \frac{\bigl| V(x)|_{x^j=y^j} - V(x)|_{x^j=z^j}\bigr|}{|y^j-z^j|^\gamma}\,.
\]
\end{defn}

\begin{rmk}
We included also $\gamma = 0$, for which $[V]_0 \leq 2\norm{V}_\infty$, where $\norm\var_\infty$ is the usual sup-norm.
\end{rmk}

\begin{defn} \label{NST_defnormb}
Given $\beta \in (\R_+)^\N$ and $\alpha \in \N^\N$ with $|\alpha| = \sum_{i \in \N} \alpha^i < \infty$, for $V \colon \R^\omega \to \R$  and $\gamma \in [0,1)$ we define
\[
\norm{V}_{\infty;\beta,\alpha} \defeq \frac{\norm{V}_\infty}{\beta^\alpha} \quad\text{and}\quad [V]_{\gamma;\beta,\alpha} \defeq \frac{[V]_\gamma}{\beta^\alpha}\,,
\]
where $\beta^\alpha$ is recursively defined as follows: if $|\alpha|=0$, $\beta^\alpha \defeq 1$, and for $|\alpha| \geq 1$,
\[
\beta^\alpha \defeq \biggl(\, \prod_{i \in \N} (\beta^i)^{\alpha^i} \biggr)^{\frac1{|\alpha|}} \wedge \min \beta^{\alpha'},
\]
the minimum being taken over all $\alpha'\leq\alpha$ (according to the lexicographic order) with $|\alpha'| = |\alpha|-1$.
\end{defn}

\begin{ex} \label{NST_exa2}
We will essentially use the above definition of $\beta^\alpha$ for $|\alpha| \leq 2$. For example, note that if $\alpha^j = 1 = \alpha^k$ and $\alpha^i = 0$ for all $i \notin \{j,k\}$, then $\beta^\alpha = \beta^j \wedge \beta^k \wedge \sqrt{\beta^j\beta^k}$.
\end{ex}

\begin{rmk} \label{NST_rmkomitb0}
One has that $\norm{V}_{\infty;\beta,0} = \norm{V}_\infty$ and $[V]_{\gamma;\beta,0} = [V]_\gamma$ are in fact independent of $\beta$. Similarly, with $1$ denoting the all-ones sequence, $\norm{V}_{\infty;1,\alpha} = \norm{V}_\infty$ and $[V]_{\gamma;1,\alpha} = [V]_\gamma$.
\end{rmk}

\begin{rmk} \label{NST_rmkbinc}
Since $\beta^\alpha$ is non-increasing with respect to $\alpha$, one has $\norm{V}_{\infty;\beta,\alpha} \leq \norm{V}_{\infty;\beta,\alpha'}$ if $\alpha \leq \alpha'$.
\end{rmk}

\begin{rmk} \label{NST_rmkHcD}
Let $\alpha \in \N^\N$ with $|\alpha| < \infty$, and suppose that $\norm{D^{\alpha'} V}_{\infty;\beta,\alpha'}$ is finite for all $\alpha'\geq\alpha$ with $|\alpha'| = |\alpha|+1$ (we are using the standard multi-index notation for derivatives; that is, $D^\alpha V = D^{\alpha^0}_0D^{\alpha^1}_1\cdots D^{\alpha^\ell}_\ell V$, with $\ell \defeq \max\{i \in \N :\ \alpha^i \neq 0\}$). Then 
\[ \begin{split}
[D^\alpha V]_{\gamma;\beta,\alpha} &\leq \frac{\sup_j \norm{D_j D^\alpha V}_\infty \vee 2\norm{D^\alpha V}_\infty}{\beta^\alpha} \\
&\leq \sup_{\substack{\alpha'\geq\alpha \\ |\alpha'| = |\alpha|+1}}\!\! \norm{D^{\alpha'}V}_{\infty;\beta,\alpha}\, \vee\, 2\norm{D^\alpha V}_{\infty;\beta,\alpha}\,.
\end{split}
\]
\end{rmk}

\begin{defn} \label{NST_defHS}
Given $m \in \N$, $\gamma \in [0,1)$ and $\beta \in (\R_+)^\N$, the space $C^{m+\gamma}_\beta(\R^\omega)$ is that of all functions $V \colon \R^\omega \to \R$ such that, for all $\alpha \in \N^\N$ with $|\alpha| \leq m$, the derivative $D^\alpha V$ exists and it is continuous, with finite norm
\[
\norm{V}_{m+\gamma;\beta} \defeq \sum_{k \leq m} \sup_{|\alpha| = k} \norm{D^\alpha V}_{\infty;\beta,\alpha} + \sup_{|\alpha| = m} [D^\alpha V]_{\gamma;\beta,\alpha}.
\]
For $\gamma = 1$, we denote by $C^{m,1}_\beta(\R^\omega)$ the space of all functions as above, with finite norm
\[
\norm{V}_{m,1;\beta} \defeq \sum_{k \leq m} \sup_{|\alpha| = k} \norm{D^\alpha V}_{\infty;\beta,\alpha} + \sup_{|\alpha| = m} [D^\alpha V]_{1;\beta,\alpha}.
\]
\end{defn}

We will also benefit of a slight variant of the above spaces, which is given as follows.

\begin{defn} \label{NST_defHS'}
Given $m \in \N \setminus \{0\}$, $\gamma \in [0,1)$ and $\beta \in (\R_+)^\N$, the space $C^{m+\gamma-}_\beta(\R^\omega)$ is that of all functions $V \in C_\beta^{m-1+\gamma}(\R^\omega)$ such that, for all $\alpha \in \N^\N$ with $|\alpha| = m$, the derivative $D^\alpha V$ exists and it is continuous, with finite norm
\[
\norm{V}_{m+\gamma-;\beta} \defeq \norm{V}_{m-1+\gamma;\beta} + \sup_{|\alpha|=m} \sup_{\substack{\alpha'\leq\alpha \\ |\alpha'| = m-1}} \bigl( \norm{D^\alpha V}_{\infty;\beta,\alpha'} + [D^\alpha V]_{\gamma;\beta,\alpha'} \bigr).
\]
One also defines $C^{m-,1}_\beta(\R^\omega)$ in an analogous manner.
\end{defn}

\begin{rmk} \label{NST_rmksubs}
By \Cref{NST_rmkbinc} (and the fundamental theorem of calculus) one has that $C^{m+\gamma}_\beta(\R^\omega)$ is a closed subspace of $C^{m+\gamma-}_\beta(\R^\omega)$, and by \Cref{NST_rmkHcD}
\[
\sum_{k \leq m-1} \sup_{|\alpha|=k} \norm{D^\alpha V}_{\infty;\beta,\alpha} + \sup_{|\alpha|=m} \sup_{\substack{\alpha'\leq\alpha \\ |\alpha'| = m-1}} \bigl( \norm{D^\alpha V}_{\infty;\beta,\alpha'} + [D^\alpha V]_{\gamma;\beta,\alpha'} \bigr)
\]
is an equivalent norm on $C^{m+\gamma-}_\beta(\R^\omega)$.
\end{rmk}

In the following, in light of \Cref{NST_rmkomitb0}, for $\gamma \in [0,1)$ we will simply write $C^\gamma(\R^\omega)$ in lieu of $C^\gamma_\beta(\R^\omega)$ and omit the subscripts $\beta$ and $\alpha$ in the seminorms; similarly $C^{0,1}(\R^\omega) \defeq C^{0,1}_\beta(\R^\omega)$. We will also omit the $\beta$ when it is the all-ones sequence; that is, $C^{m+\gamma}(\R^\omega) \defeq C^{m+\gamma}_{{1}}(\R^\omega)$. The same conventions are understood for the spaces given in \Cref{NST_defHS'} as well.

%\begin{rmk} \label{NST_rmkC0Cinfty}
%According to the above definitions, $\norm{\var}_{\infty} \leq \norm{\var}_{0} \leq 3\norm{\var}_{\infty}$.
%\end{rmk}

\begin{lem} \label{NST_CmgB}
The spaces $C^{m+\gamma-}_\beta(\R^\omega)$ and $C^{m+\gamma}_\beta(\R^\omega)$ are Banach.
\end{lem}

\begin{proof}
We prove that $C^{\gamma}(\R^\omega)$ is Banach, and then that so is $C^{m+\gamma-}_\beta(\R^\omega)$ for $m\geq1$. This will give also the completeness of $C^{m+\gamma}_\beta(\R^\omega)$ due to \Cref{NST_rmksubs}.

If $(V^N)_{N\in \N} \subset C^{\gamma}(\R^\omega)$ is Cauchy, by the Ascoli--Arzelà theorem there exists a subsequence $V^{N_n} \to V \in C^0(\R^\omega)$ as $n \to \infty$ in the topology of compact convergence.\footnote{That is, $V^{N_n} \to V$ uniformly on each compact subset of $\R^\omega$; cf., e.g., \cite[Theorem~7.18]{keltop}.} We have, for any $x,y \in \R^\omega$ with $x^i = y^i$ for all $i \neq j$,
\[ \begin{split}
&\bigl|V(x) - V^N(x) - \bigl(V(y) - V^N(y)\bigr)\bigr| \\
&\hspace{50pt} \leq \limsup_{n\to\infty}\, \bigl|V^{N_n}(x) - V^{N_n}(y) - \bigl(V^N(x) - V^N(y)\bigr)\bigr| \\
&\hspace{50pt} \leq |x^j-y^j|^\gamma \limsup_{n\to\infty}\, [V^{N_n} - V^N]_{\gamma}\,,
\end{split}
\]
and $\norm{V - V^N}_{\infty} \leq \limsup_{n\to\infty} \norm{V^{N_n} - V^N}_{\infty}$. Therefore $V^N \to V$ in $C^\gamma(\R^\omega)$.

Fix now a bijection $\hat\alpha \colon \N \to \{ \alpha \in \N^\N :\ |\alpha| \leq m\}$ and let $\call Y$ be the space of sequences of functions $W \in \ell^\infty(\N;C^\gamma(\R^\omega))$ with finite norm
\[
\norm{W}_{\call Y} \defeq \sum_{k \leq m} \,\sup_{i \in N :\, |\hat\alpha(i)| = k} \frac{\norm{W^i}_\infty + [W^i]_\gamma}{w(\hat\alpha(i))}\,,
\]
where
\[
w(\hat\alpha(i)) \defeq \begin{dcases}
\beta^{\hat\alpha(i)} & \text{if} \ |\hat\alpha(i)| < m \\
\min_{\substack{\alpha \leq \hat\alpha(i) \\ |\alpha'| = m-1}} \beta^{\alpha} & \text{if} \ |\hat\alpha(i)| = m\,;
\end{dcases}
\]
$\call Y$ is easily seen to be a Banach space. Then consider the linear map $\imath \colon C^{m+\gamma}_\beta(\R^\omega) \to \call Y$ given by $\imath(V)^i \defeq D^{\hat\alpha(i)} V$, for $V \in C^{m+\gamma}_\beta(\R^\omega)$ and $i \in \N$
By \Cref{NST_rmkHcD}, $\norm{V}_{m+\gamma-;\beta,\alpha} \leq \norm{\imath(V)}_{\call Y} \leq 3 \norm{V}_{m+\gamma-;\beta}$.
Therefore, if $(V^N)_{N\in\N} \subset C^{m+\gamma-}_\beta(\R^\omega)$ is Cauchy, then $\imath(V^N) \to W$ in $\call Y$. By compact convergence, if $|\alpha(j)| = |\alpha(i)| + 1$ with $\alpha(j)^k = \alpha(i)^k+1$, then $W^i (x) - W^i(x)|_{x^k = 0}
= \int_0^{x^k} W^j(x)|_{x^k = y} \,\di y$.
Exploiting the fundamental theorem of calculus, one can prove by induction that this implies that $W^i = D^{\alpha(i)} W^{\alpha^{-1}(0)}$; therefore, $W = \imath(W^{\alpha^{-1}(0)})$. By the above equivalence of norms and the linearity of $\imath$, we conclude that $V^N \to W^{\alpha^{-1}(0)}$ in $C^{m+\gamma-}_\beta(\R^\omega)$.
\end{proof}

In a similar fashion, we define parabolic (H\"{o}lder) spaces on $\R^\omega_T$.

\begin{defn}
For $V \colon \R^\omega_T \to \R$ and $\gamma \in [0,1)$ we define
\[ \begin{split}
[V]_{\frac\gamma2,\gamma} &\defeq \sup_{x \in \R^\omega} [V(\var,x)]_{\frac\gamma2} + \sup_{t \in [0,T]} [V(t,\var)]_\gamma \\
&= \sup_{s \neq t} \frac{\norm{V(t,\var) - V(s,\var)}_\infty}{|t-s|^{\frac\gamma2}} + \sup_{t \in [0,T]} [V(t,\var)]_\gamma\,.
\end{split}
\]
Then, given $\beta \in (\R_+)^\N$ and $\alpha \in \N^\N$ with $|\alpha| < \infty$, for $V \colon \R^\omega_T \to \R$ we define
\[
[V]_{\frac\gamma2,\gamma;\beta,\alpha} \defeq \sup_{x \in \R^\omega} \frac{[V(\var,x)]_{\frac\gamma2}}{(\beta^\alpha)^{\frac12}} + \sup_{t \in [0,T]} \frac{[V(t,\var)]_\gamma}{\beta^\alpha}\,,
\]
where $\beta^\alpha$ is defined as in \Cref{NST_defnormb}.
\end{defn}

\begin{defn}
Given $\gamma \in [0,1)$, the space $C^{\frac\gamma2,\gamma}(\R^\omega_T)$ is that of all continuous functions $V \colon \R^\omega_T \to \R$ with finite norm
\[
\norm{V}_{\frac\gamma2,\gamma} \defeq \norm{V}_\infty + [V]_{\frac\gamma2,\gamma}\,.
\]
Given also $\beta \in (\R_+)^\N$, the space $C^{1,2}_\beta(\R^\omega_T)$ is that of all functions $V \colon \R^\omega_T \to \R$ such that the derivatives $\de_t V$ and $D^\alpha V$, for all $\alpha \in \N^\N$ with $|\alpha| \leq 2$ exists and are continuous on $(0,T) \times \R^\omega$, with finite norm
\[
\norm{V}_{1,2;\beta} \defeq \norm{\de_t V}_\infty + \sum_{k \leq 2} \sup_{|\alpha| = k} \norm{D^\alpha V}_{\infty;\beta,\alpha}\,. %+ [\de_t V]_{\frac\gamma2,\gamma} + \sup_{|\alpha| = 2} \sup_{\substack{\alpha' \leq \alpha \\ |\alpha'|=1}} [D^\alpha V]_{\frac\gamma2,\gamma;\beta,\alpha'}.
\]
\end{defn}

With a proof very much similar to that of \Cref{NST_CmgB} (which we thus omit), we can also show the following fact.

\begin{lem} %\label{NST_CmgB'}
The spaces $C^{\frac\gamma2,\gamma}(\R^\omega_T)$ and $C^{1,2}_\beta(\R^\omega_T)$ are Banach.
\end{lem}

We can now set the notion of classical solution to the infinite-dimensional Nash system. The space $C^{0}(\R^\omega_T)$ appearing below is defined in the usual way as the space of all continuous functions from $\R^\omega_T$ to $\R$ with bounded norm $\norm{\var}_\infty$; alternatively, it can be understood as the space $C^{0,0}(\R^\omega_T)$ defined above, as the two respective norms are equivalent.

\begin{defn} \label{NST_csoldef}
We will say that $u = (u^i)_{i\in\N}$ is a classical solution to the Nash system \eqref{tdli_ns} if the following happens:
\begin{itemize}[leftmargin=2em]
\item for each $i \in \N$, $u^i \in C^{1,2}_{\beta_i}(\R^\omega_T)$ for some $\beta_i \in (\R_+)^\N$, and $\sup_{i\in\N} \norm{u^i}_{1,2;\beta_i} < \infty$;
\item the series appearing in the equations converge in $C^0(\R^\omega_T)$ uniformly in $i$;
\item the equations are satisfied pointwise in $\R^\omega_T$.
\end{itemize}
\end{defn}

%\begin{rmk} \label{NST_rmkscs}
%Besides the convergence requirements, one need not ask that $u$ is uniformly bounded in $i$ for a meaningful definition of a classical solution to \eqref{tdli_ns}. Nevertheless, it seems very reasonable to ask a classical solution to satisfy such a stronger property if one thinks to the infinite-dimensional Nash system as a sort of limit of the $N$-dimensional one as $N$ diverges, as a stable convergence of the whole system is expected to be uniform with respect to the index $i$, as its equations are all paired.
%\end{rmk}

\section{A priori estimates on linear parabolic equations} \label{NST_secape}

We begin with an important general estimate on the derivatives of a solution to a Fokker--Planck--Kolmogorov equation. We state this result under slightly more general hypotheses than the ones under which we will use it, for the convenience of possible further applications. 

\begin{lem} \label{tdli_lem:kfp}
Let $N \in \N$. Let $A \colon \R \times \R^N \to \Sym(N)$ and $B \colon \R \times \R^N \to \R^N$ be bounded and continuous, with $D_i A^{hk}$, $D^2_{ij}A^{hk}$ and $D_iB^j$ bounded for all $i,j,h,k \in \DSet N$.
Suppose that $A \geq \lambda I$ for some $\lambda>0$ and $D_k A^{ij} = 0$ if $i \neq k \neq j$.
Let $\rho = \rho_\epsilon \in C^{1,2}((s,T)\times\R^N) \cap C^0([s,T) \times \R^N)$ solve
\begin{equation} \label{tdli_kfp} \begin{dcases}
\de_t \rho - \sum_{j,k \in \DSet N} D^2_{ij}(A^{ij}\rho) - \mathrm{div}(B\rho) = 0 & \text{in}\ (s,T) \times \R^N \\
\rho(s) = \delta_y \star \eta_\epsilon,
\end{dcases}\end{equation}
for $T>0$ and some $(s,y) \in [0,T) \times \R^N$, with $(\eta_\epsilon)_{\epsilon > 0}$ being an approximation of the identity as $\epsilon \to 0$.
Then there exists $T^*>0$, depending only on
\begin{gather*}
\lambda, \quad \sup_i \,\biggl\lVert \,\sum_j |D_iB^j|^2 \biggr\rVert_\infty\,, \\
\sup_i \,\biggl\lVert\, \sum_j D_{i} A^{ij} \biggr\rVert_\infty\,, \quad \sup_i\, \biggl\lVert\, \sum_j D^2_{ij} A^{ij} \biggr\rVert_\infty\,, \quad \text{and} \quad \sup_i\, \biggl\lVert \, \sum_j \Bigl\lvert  \sum_k D^2_{ik} A^{jk} \Bigr\rvert^2 \biggr\rVert_\infty\,,
\end{gather*}
and $C>0$, depending only on the above quantities and $T^*$,
such that, if $T < T^*$, then
\begin{equation} \label{tdli_estdkr}
\lim_{\epsilon \to 0} \sup_k \int_s^t \! \int_{\R^N} |D_k\rho_\epsilon| \leq C \sqrt{t-s} \qquad \forall\, t \in [s,T].
\end{equation}
\end{lem}

\begin{proof}
%A solution $\rho$ as in the statement exists by, e.g., \cite[Theorem~6.6.1]{bkrs}, which also tells that $\rho(t,x) \to 0$ as $|x| \to \infty$.
%Some of the following computations are formal; nevertheless they can be justified by standard approximation arguments and classic regularity results that can be found, e.g., in \cite{bkrs,kryHS}.
By the results in \cite[Chapter~1]{fried}, we know that \eqref{tdli_estdkr} holds for some constant $C$ and any $T > 0$. We want to make explicit the dependence of $C$ on the data, for $T$ near $0$.
Let $\tau \in (s,T)$ and consider a sequence of smooth functions $\psi^{\tau,k}_{\epsilon,n} \to \mathrm{sgn}\,D_k \rho_{\epsilon}(\tau)$ in $L^1_\loc(\R^N)$ as $n \to \infty$, with $\norm{\psi^{\tau,k}_{\epsilon,n}}_\infty = 1$ for all $n$. Since $\int_{\R^N} D_k \rho_\epsilon(\tau)\psi^{\tau,k}_{\epsilon,n}(\tau) \to \int_{\R^N} |D_k \rho_\epsilon(\tau)|$ as $n \to \infty$, for any sequence $\epsilon_n \to 0$ we can consider $\ell_n \to \infty$ such that
\[
\int_{\R^N} \bigl| D_k \rho_n(\tau)\psi^{\tau,k}_{n}(\tau) - |D_k \rho_n(\tau)| \bigr| \to 0 \quad \text{as} \ n \to \infty,
\]
where $\rho_n \defeq \rho_{\epsilon_n}$ and $\psi^{\tau,k}_n \defeq \psi_{\epsilon_n,\ell_n}$.
Let now $w = w^{\tau,k}_{n}$ solving
\begin{equation} \label{tdli_bke}
-\de_t w - \tr(AD^2 w) + \pair{B}{Dw} = 0 \quad \text{on}\ (s,\tau) \times \R^N,
\end{equation}
with terminal condition $w|_{t=\tau} = \psi_n = \psi^{\tau,k}_{n}$. %\star \tilde\eta$ for some $\tilde\eta \in C^\infty_c(\R^N) \cap \call P(\R^N)$.
We first notice that testing the equation of $\rho_\epsilon$ by $\frac12 w^2$ over $(s,\tau)$ we get
\begin{equation} \label{tdli_estw2}
(w(s,\cdot)^2 \star \eta_{\epsilon})(y) + 2\int_s^{\tau} \int_{\R^N} \pair{ADw}{Dw} \rho_\epsilon \leq 1.
\end{equation}
On the other hand, $v(t,x) \defeq \frac{\tau-t}{2} |D_kw(t,x)|^2$ satisfies
\[ \begin{multlined}[.95\displaywidth]
-\de_t v - \tr(AD^2v) + \pair{B}{Dv} \\
\leq \frac12 |D_k w|^2 - (\tau-t)\pair{B_k}{Dw}D_k w + (\tau-t)\tr(D_k A D^2 w)D_k w,
\end{multlined}
\]
so testing by $\rho_{\epsilon}$ over $(s, \tau)$ one obtains
\begin{equation} \label{tdli_estDk1}
\begin{split}
&\int_{\R^N} v(s,\var) \eta_{\epsilon}(y - \var) \\
&\qquad\qquad \leq \frac12 \int_s^\tau \int_{\R^N} |D_k w|^2 \rho_\epsilon - \int_s^\tau \int_{\R^N} (\tau-t) \pair{D_kB}{Dw} D_k w \rho_\epsilon \\
&\quad\qquad\qquad + \int_s^\tau \int_{\R^N} (\tau-t)\tr(D_k A D^2 w)D_k w \rho_\epsilon\,.
\end{split}
\end{equation}
Using estimate \eqref{tdli_estw2} and the ellipticity assumption $A \geq \lambda I$, along with Young's and the Cauchy--Schwarz inequalities, we estimate
\[
\frac12 \int_s^\tau \int_{\R^N} |D_k w|^2 \rho_\epsilon - \int_s^\tau \int_{\R^N} (\tau-t) \pair{D_kB}{Dw} D_k w \rho_\epsilon \leq \frac{1 + \tau (1+C_B) }{4\lambda},
\]
with $C_B \defeq \norm{D_kB}_{2,\infty}^2$.
On the other hand, using our assumptions on $A$ and integrating by parts, 
\[ \begin{split}
\int_s^\tau \int_{\R^N} (\tau-t)\tr(D_k A D^2 w)D_k w \rho_\epsilon &= - \frac12 \int_s^\tau \int_{\R^N} (\tau-t) \sum_j D^2_{jk} A^{jk} |D_k w|^2 \rho_\epsilon \\
&\quad - \frac12 \int_s^\tau \int_{\R^N} (\tau-t) \sum_j D_k A^{jk} |D_k w|^2 D_j \rho_\epsilon\,,
\end{split}
\]
whence
\[ \begin{split}
&\int_s^\tau \int_{\R^N} (\tau-t)\tr(D_k A D^2 w)D_k w \rho_\epsilon \\
&\qquad\qquad \leq \frac{T C_A}{4\lambda} + C_A' \sup_{t \in [s,\tau]} \norm{v(t,\var)}_\infty \,\sup_j \int_s^\tau \!\int_{\R^N}|D_j \rho_\epsilon|\,,
\end{split}
\]
with $C_A \defeq \sup_k\, \bigl\lVert\, \sum_j D^2_{jk} A^{jk} \bigr\rVert_\infty$ and  $C_A' \defeq \sup_k \,\bigl\lVert\, \sum_j D_{k} A^{jk} \bigr\rVert_\infty$. 
Let now $y=z$ with $z = z(s,N,\tau,n)$ such that $|D_k w(s,z)|^2 \geq \norm{D_k w(s,\var)}_\infty^2 - 1$ for all $k \in \DSet N$. We will write $\rho_\epsilon=\rho^z_\epsilon$ to stress that $\rho_\epsilon^z(s)=\delta_z \star \eta_\epsilon$. Then for $\epsilon_n > 0$ so small that $\int_{\R^N} v(s,\var)\eta_{\epsilon_n}(z-\var) \geq v(s,z) - 1$, from the previous estimates we have
\begin{equation} \label{NST_lemrhoest1}
\begin{split}
\norm{v(s,\cdot)}_\infty 
&\leq 1 + \frac{T}2 + \frac{1 + T (1+ C_B + C_A)}{4\lambda} \\
&\quad + C_A' \sup_{t \in [s,\tau]} \norm{v(t,\var)}_\infty\, \sup_j \int_s^\tau \!\int_{\R^N}|D_j \rho^z_n|\,.
\end{split}
\end{equation}
By continuity, for any $\tau$ close enough to $s$ (say $s < \tau < \hat\tau = \hat\tau(s,n,N) \leq T$) one has
\[
\sup_j \int_s^{\tau} \!\int_{\R^N}|D_j \rho^z_n| \leq \frac{1}{2C_A'},
\]
so \eqref{NST_lemrhoest1} yields
\[
\norm{v(s,\cdot)}_\infty \leq 2 + T + \frac{1 + T (1+ C_B + C_A)}{2\lambda}.
\]
This implies that for all $\tau \leq \bar\tau \defeq \min_{t\in[s,T]} \hat\tau(t,n,N)$ one has 
\begin{equation} \label{NST_lemrhoest2}
\sup_k\, \norm{D_k w^{\tau,k}_n(t,\cdot)}_\infty \leq \frac{\bar C}{\sqrt{\tau-t}} \quad \forall\, t \in [s,\tau],
\end{equation}
with $\bar C = \bar C(\lambda,T,C_B,C_A)$ being the square root of the constant above. We want to show that we can choose $\bar\tau \equiv T$, provided that $T$ is small enough.
Testing the equation for of $w$ by $D_k\rho_n$ over $(s,\tau)$ we get 
\[ \begin{split}
- \int_{\R^N} D_k \rho_n(\tau) \psi_{n}(\tau) &= \int_{\R^N} D_k w(s,\cdot) \rho_n(s) +\int_s^\tau \! \int_{\R^N} \pair{D_k B}{Dw}\rho_n \\
&\quad + \sum_{ij} \int_s^\tau \!\int_{\R^N} D^2_{ik} A^{ij} D_{j} w \rho_n
+ \sum_{ij} \int_s^\tau \!\int_{\R^N} D_k A^{ij} D_j w D_i \rho_n,
\end{split}
\]
whence, by the Cauchy--Schwarz and Young's inequalities, and estimate~\eqref{tdli_estw2},
\begin{equation} \label{NST_lemrhoest3}
\begin{split}
\int_{\R^N} D_k \rho_n(\tau)\psi_{n}(\tau) &\leq \norm{D_k w(s,\cdot)}_\infty + T\Bigl( \frac12 C_B + C_A'' \Bigr) + \frac{1}{2\lambda} \\
&\quad + C_A'\, \sup_i \int_s^\tau \sup_j\, \norm{D_j w(t,\cdot)}_{\infty} \int_{\R^N} |D_i \rho_n|(t)\,\di t\,,
\end{split}
\end{equation}
where $C_A'' \defeq \sup_k\, \bigl\lVert \sum_i D^2_{ik} A^{i\var} \bigr\rVert_{2,\infty}^2$.
Let $n_0 = n_0(\tau,N) \in \N$ be such that
\[
\int_{\R^N} D_k \rho_n(\tau)\psi^{\tau,k}_{n}(\tau) \geq \int_{\R^N} |D_k \rho_n(\tau)| - 1
\]
for all $n \geq n_0$ and for all $k \in \DSet N$; for those $n$, by Gronwall's lemma,
\[ \begin{split}
&\sup_k \int_{\R^N} |D_k\rho_n|(\tau) \\
&\qquad \leq \bigl( 1 + \sup_k\, \norm{D_k w_{n}^{\tau,k}(s,\var)}_{\infty} + \hat C \bigr) \exp\biggl\{ C_A' \int_s^\tau \sup_j \,\norm{D_j w_{n}^{\tau,k}(t,\var)}_{\infty}\,\di t\biggr\}\,,
\end{split}
\]
where $\hat C \defeq T\bigl( \frac12 C_B + C_A'' \bigr) + \frac{1}{2\lambda}$. For $\tau \leq \bar\tau$ as above, we obtain
\begin{equation} \label{NST_lem1le}
\begin{split}
\sup_k \int_s^\tau \int_{\R^N} |D_k\rho_n| &\leq \sqrt{\tau-s}\,\bigl( \sqrt T(1+\hat C) + \bar C \bigr) e^{C_A'\bar C T} \\
&\leq \sqrt{T}\bigl(\sqrt{T}(1+\hat C) + \bar C\bigr) e^{C_A'\bar CT}\,,
\end{split}
\end{equation}
which holds for any $\rho_n = \rho_n^y$; in particular, if $\sqrt{T}\bigl(\sqrt{T}(1+\hat C) + \bar C\bigr)e^{C_A'\bar CT} \leq \frac1{2C_A'}$ and we let $y = z$ as above, by a continuity argument we see that we can choose $\hat\tau \equiv T$, and thus $\bar\tau = T$. We deduce that estimate~\eqref{NST_lem1le} in fact holds for all $\tau \in [s,T]$ and we can let $n \to \infty$ to get \eqref{tdli_estdkr}.
\end{proof}

The following result provides a crucial decay estimates for the derivatives of a solution to a linear transport-diffusion equation whose differential operator is the adjoint of the one of the FPK equation.

\begin{prop} \label{tdli_proppret}
Let $N \in \N$. Consider $A$ and $B$ as in \Cref{tdli_lem:kfp}, and let $G \in C^3(\R^N)$ and $F \in C^{0}([0,T];C^2(\R^N))$.
Suppose that $D_k A^{ij} = 0$ unless $i=j=k$.
Let $\beta \in \ell^{1}(\bb Z;\R_+)$ be even and such that $\beta \star \beta \leq c\beta$; that is,
\[
\sum_{j \in \Z} \beta^j \beta^{i-j} \leq c \beta^i \quad \forall\, i \in \Z. 
\]
Suppose there exist constants $c_B,c_F,c_G\geq0$ such that, for all $i,j,k,l \in \DSet{N}$,
\begin{align}
\label{tdli_e1} {\norm{D_j B^{i}}_\infty} &\leq c_B{\beta^{j-i}}\,, &  {\norm{D_{jk} B^{i}}_\infty} &\leq c_B( \beta^{j-i} \wedge \sqrt{\beta^{j-i}\beta^{k-i}}\,)\,, \\
\label{tdli_e3}  {\norm{D_j F}_\infty} &\leq c_F \beta^j, & \norm{D_{jk}^2 F}_\infty &\leq c_F ({\beta^j} \wedge \sqrt{\beta^j\beta^k}\,)\,, \\
\label{tdli_e2} {\norm{D_j G}_\infty} &\leq c_G \beta^j\,, & \norm{D^2_{jk} G}_\infty \vee \norm{D^3_{jkl} G}_\infty& \leq c_G ({\beta^j} \wedge \sqrt{\beta^j\beta^k}\,)\,.
\end{align}
Then a classical solution $w \in C^0([0,T];C^3(\R^N))$ to
\begin{equation} \label{tdli_dualkfp}
\begin{cases}
-\de_t w - \tr(AD^2 w) + \pair{B}{Dw} = F & \text{on}\ (0,T) \times \R^N \\
w|_{t=T}= G
\end{cases}
\end{equation}
satisfies
\begin{equation} \label{NST_estonw}
{\norm{D_j w}_\infty} \leq K\beta^j, \quad {\norm{D^2_{jk} w}_\infty + \norm{D^3_{jkl} w}_\infty} \leq K ({\beta^j} \wedge \sqrt{\beta^j \beta^k}\,)
\end{equation}
for all $j,k,l \in \{0,\dots,N-1\}$,
where the constant $K$ depends only on $T$, $c$, $c_B$, $c_F$, $c_G$, $\norm{\beta}_1$ and $c_A \defeq \max_{1 \leq \ell \leq 3} \sup_k \norm{(D_k)^\ell A^{kk}}_\infty$. More precisely, one can choose $K = c_G + c_F \tilde K$ with $\tilde K$ vanishing as $T \to 0$. % and depending cubically on $c_B$ and $T$ as long as $T \leq \theta c_B^{-1}$ for some constant $\theta$ independent of $c_B$.
\end{prop}

Note that recalling \Cref{NST_exa2} the above bounds \eqref{tdli_e1}, \eqref{tdli_e3} \eqref{tdli_e2} and \eqref{NST_estonw} can be expressed in terms of the norms of suitable spaces among those introduced \Cref{NST_secholsp}. This allows a more compact notation, which will be used in the following \Cref{NST_secexu}; nevertheless, here we agreed to write the estimates in a more explicit form, for the benefit of the reader who needs to get used to the kind of controls we wish to eventually have on the derivatives of the data and the solution of the Nash system.

\begin{proof}[Proof of \Cref{tdli_proppret}]
The following computations are performed assuming $w$ to be smooth; nevertheless, by means of a standard approximation argument, one can prove that the estimates we get hold for $w$ as regular as in the statement of this proposition.

\emph{First-order estimates.} Testing the equation of $D_kw$ by $\rho_\epsilon$ solving \eqref{tdli_kfp} and letting $\epsilon \to 0$, after easy computations one gets
\[ \begin{split}
&\norm{D_k w(s,\var)}_{\infty} \\
&\qquad \leq \norm{D_k G}_\infty + (T-s)\bigg( c_B \sup_i \frac{\norm{D_i w}_{\infty;[s,T]}}{\beta^i} \sum_{j \in \DSet N} \beta^{k-j} \beta^j + \norm{D_k F}_{\infty;[s,T]} \bigg) \\
&\qquad \quad + \int_s^T \biggl| \int_{\R^N} D_k A^{kk} D^2_{kk} w \rho \biggr|\,;
\end{split}
\]
here $\norm{\var}_{\infty;[s,T]} \defeq \sup_{[s,T] \times \R^N} |\var|$ and $\rho(t) \defeq \lim_{\epsilon \to 0} \rho_\epsilon(t) \in C^0(\R^d)$, and we have used \eqref{tdli_e1}. The last integral is obtained exploiting our assumption on $A$ and 
%can be estimated by using our assumption on $\Sigma$ and integrating by parts; indeed, as $A^{\ell m}$ is independent of $x^h$ if $h \notin \{\ell,m\}$, we have, with the notation $\sum_{\ell;k} \defeq \sum_{\ell=k} + 2\sum_{\ell \neq k}$ for summations,
%\[
%\int_s^T \int_{\R^N} \sum_{\ell m} D_k A^{\ell m} D^2_{\ell m} w \rho =  \int_s^T \int_{\R^N} \sum_{\ell ; k} D_k A^{\ell k} D^2_{\ell k} w \rho  
%&= -\int_s^T \int_{\R^N} \sum_{\ell ; k} D_{kk}^2 A^{\ell k} D_{\ell} w \rho - \int_s^T \int_{\R^N} \sum_{\ell ; k} D_k A^{\ell k} D_{\ell} w D_k \rho,
%\]
which is controlled by
\[
(T-s) C_A \sup_{ij} \frac{\norm{D^2_{ij} w}_{\infty;[s,T]}}{\beta^i \wedge \sqrt{\beta^i\beta^j}}\,\beta^k,
\]
with $C_A \defeq \sup_k \norm{D_k A^{kk}}_\infty$.
Then, if $T-s$ small enough, using that $\beta \star \beta \leq c\beta$ we deduce from the above estimates that
\begin{equation} \label{NST_estD1}
\sup_k \frac{\norm{D_k w}_{\infty;[s,T]}}{\beta^k} \leq c_1(T-s)\biggl(1+\sup_{ij} \frac{\norm{D^2_{ij} w}_{\infty;[s,T]}}{\beta^i \wedge \sqrt{\beta^i\beta^j}}\biggr),
\end{equation}
where $c_1(T-s)$ is an increasing function of $T-s$ that can be explicitly written using the parameters $c, C_A, c_B, c_F, c_G, \norm{\beta}_1$.
%To get the estimate on $[0,T]$ it suffices now to repeat this argument a finite number of steps with $w(T-\kappa s^*,\cdot)$ in lieu of $G$, for a suitable $s^*$ and $\kappa = 1,\dots,\lceil T/s^* \rceil$.

\textit{Second order estimates.} Testing the equation of $D_{jk}^2 w$ by $\rho$ one gets
\begin{equation} \label{NST_estD2torec}
\begin{split}
&\norm{D^2_{jk} w(s,\var)}_{\infty} \\
&\quad \leq \norm{D^2_{jk} G}_\infty + (T-s) \norm{D_{jk}^2 F}_{\infty;[s,T]} + \int_s^T \!\! \int_{\R^N} \bigl| D^2_{jk}\pair{B}{Dw} -  \pair{B}{D^2_{jk}D w} \bigr| \rho \\
&\qquad + \int_s^T \biggl| \int_{\R^N} \bigl( D^2_{jk} A^{jj} D^2_{jj} w + D_{j} A^{j j} D^3_{j jk} w + D_{k} A^{k k} D^3_{jkk} w \bigr) \rho \biggr|\,,
\end{split}
\end{equation}
where in particular $D^2_{jk} A^{jj} = 0$ if $j \neq k$. It easy to see that the last integral is controlled by
\[
3(T-s)C_A' \Bigl( \sup_{il} \frac{\norm{D^2_{il} w}_{\infty;[s,T]}}{\beta^i \wedge \sqrt{\beta^i\beta^l}} + \sup_{ilm} \frac{\norm{D^3_{ilm} w}_{\infty;[s,T]}}{\beta^i \wedge \sqrt{\beta^i\beta^l}} \Bigr) (\beta^j \wedge \sqrt{\beta^j \beta^k}\,),
\]
with $C_A' \defeq C_A + \sup_k \norm{D_{kk} A^{kk}}_\infty$.
On the other hand, using \eqref{tdli_e1} and the property $\beta \star \beta \leq c \beta$, one can estimates in two ways, according to whether one uses that $|D^2_{jk} B^i| \leq c_B\beta^j$ or $|D^2_{jk} B^i| \leq c_B\sqrt{\beta^j\beta^k}$: we have either
\[ \begin{split}
&\int_{\R^N} \bigl| D^2_{jk}\pair{B}{Dw}
-  \pair{B}{D^2_{jk}D w} \bigr| \rho \\
& \qquad \leq c_B \biggl(\, \sup_i \frac{\norm{D_i w}_{\infty;[s,T]}}{\beta^i} \cdot \sum_i \beta^{j-i} \beta^i \\
& \qquad \qquad \quad + \sup_{ij} \frac{\norm{D^2_{ij} w}_{\infty;[s,T]}}{\beta^i} \cdot \Bigl( \norm{\beta}_1 \beta^j + \sum_i \beta^{j-i} \beta^i \Bigr) \biggr) \\
& \qquad \leq c_B \biggl( c \sup_i \frac{\norm{D_i w}_{\infty;[s,T]}}{\beta^i} + (\norm{\beta}_1 +c) \sup_{ij} \frac{\norm{D^2_{ij} w}_{\infty;[s,T]}}{\beta^i} \biggr) \beta^j,
\end{split}
\]
or, exploiting the Cauchy-Schwarz inequality,
\[ \begin{split}
&\int_{\R^N} \bigl| D^2_{jk}\pair{B}{Dw} -  \pair{B}{D^2_{jk}D w} \bigr| \rho \\
& \qquad \leq c_B \sup_i \frac{\norm{D_i w}_{\infty;[s,T]}}{\beta^i} \cdot \sum_i (\beta^{j-i})^{\frac12} (\beta^{k-i})^{\frac12} \beta^i \\
& \qquad \quad + c_B \sup_{ij} \frac{\norm{D^2_{ij} w}_{\infty;[s,T]}}{\sqrt{\beta^i \beta^j}} \cdot \Bigl( \,\sum_i \beta^{k-i} (\beta^{j})^{\frac12} (\beta^i)^{\frac12} + \sum_i \beta^{j-i} (\beta^{k})^{\frac12} (\beta^i)^{\frac12} \Bigr) \\
& \qquad \leq \biggl( c c_B \sup_i \frac{\norm{D_i w}_{\infty;[s,T]}}{\beta^i} + 2c_B\sqrt{c \norm{\beta}_1} \sup_{ij} \frac{\norm{D^2_{ij} w}_{\infty;[s,T]}}{\sqrt{\beta^i \beta^j}} \biggr) \sqrt{\beta^j \beta^k}\,.
\end{split}
\]
Therefore, also using \eqref{NST_estD1}, we deduce that, for small $T-s$,
\begin{equation} \label{NST_estD2}
\sup_{jk} \frac{\norm{D^2_{jk} w}_{\infty;[s,T]}}{\beta^j \wedge \sqrt{\beta^j\beta^k}} \leq c_2(T-s) \biggl(1+\sup_{ijk} \frac{\norm{D^2_{ijk} w}_{\infty;[s,T]}}{\beta^i \wedge \sqrt{\beta^i\beta^j}}\biggr),
\end{equation}
where $c_2(T-s)$ is an increasing function of $T-s$ depending also on the parameter $C_A'$ and those listed for $c_1$.

\textit{Third-order estimates.} Testing the equation of $D_{jkl}^3 w$ by $\rho_\epsilon$ one gets
\[ \begin{split}
&\norm{D^3_{jkl} w(s,\var)}_{\infty} \leq \norm{D^3_{jkl} G}_\infty \\
&\quad + \lim_{\epsilon\to0} \biggl( \biggl| \int_s^T \int_{\R^N} \bigl( D^3_{jkl}\pair{B}{Dw} -  \pair{B}{D^3_{jkl}D w} \bigr) \rho_\epsilon \biggr| + \biggl| \int_s^T \int_{\R^N} D_{jkl}^3 F \rho_\epsilon \biggr| + \call A_\epsilon \biggr),
\end{split}
\]
where $\call A_\epsilon$ collects all terms involving derivatives of $A$. By analogous estimates as above, it is not difficult to see that all terms in $\call A$ where no fourth-order derivatives of $w$ appear can be controlled by
\[
(T-s)C_A''\norm{\beta}_\infty \Bigl( \sup_{il} \frac{\norm{D^2_{il} w}_{\infty;[s,T]}}{\beta^i \wedge \sqrt{\beta^i\beta^l}} + \sup_{ilm} \frac{\norm{D^3_{ilm} w}_{\infty;[s,T]}}{\beta^i \wedge \sqrt{\beta^i\beta^l}} \Bigr) (\beta^j \wedge \sqrt{\beta^j \beta^k}\,),
\]
with $C_A'' \defeq C_A' + \sup_k \norm{D_{kkk} A^{kk}}_\infty$.
This excludes three terms, which are of the form
\[
\int_s^T \biggl| \int_{\R^N} \! D_j A^{jj} D^4_{jjkl} w \rho_\epsilon \biggr| \leq \int_s^T \biggl| \int_{\R^N} \! D^2_{jl} A^{jj} D^3_{jjk} w \rho_\epsilon \biggr| + \int_s^T \biggl| \int_{\R^N} \! D_j A^{jj} D^3_{jjk} w D_l \rho_\epsilon \biggr|
\]
for any permutation of $j,k,l$, and thus, using also \Cref{tdli_lem:kfp}, as $\epsilon \to 0$ they are controlled by
\[
\sqrt{T-s}\, \bigl(C_A' + \sqrt{T-s}\, C_A'' C \bigr) \sup_{ilm} \frac{\norm{D^3_{ilm} w}_{\infty;[s,T]}}{\beta^i \wedge \sqrt{\beta^i\beta^l}} (\beta^j \wedge \sqrt{\beta^j \beta^k}\,).
\]
On the other hand, we have
\begin{equation} \label{NST_aboverhs}
\begin{split}
& \biggl| \int_{\R^N} \bigl( D^2_{jkl}\pair{B}{Dw} -  \pair{B}{D^3_{jkl}D w} \bigr) \rho_\epsilon(t,\cdot) \biggr| \\
& \qquad \leq \sum \bigl( \norm{\pair{D_j B}{D^2_{kl} Dw}}_{\infty;[s,T]} + \norm{\pair{D_{jk}^2 B}{D_{l} Dw}}_{\infty;[s,T]} \bigr) \\
& \qquad \quad - \norm{\pair{D_{jk}^2 B}{D_{l} Dw}}_{\infty;[s,T]} + \norm{\pair{D^2_{jk} B}{Dw}}_{\infty;[s,T]} \int_{\R^N} \!|D_l \rho_\epsilon(t,\cdot)|\,,
\end{split}
\end{equation}
where the sums are over cyclic permutations of the indices $j,k,l$. The basic idea to derive \eqref{NST_aboverhs} is using integration by parts to move a derivative from $B$ to $\rho_\epsilon$ when otherwise the order of the derivative applied to $B$ would be greater than $2$. Then using \eqref{tdli_e1}, the property $\beta \star \beta \leq c \beta$ and \Cref{tdli_lem:kfp}, one can show that the integral of the right-hand side of \eqref{NST_aboverhs}, as $\epsilon\to0$, is controlled either by
\[
\begin{split}
\sqrt{T-s}\, c_B \biggl( &\sqrt{T-s}\,(\norm{\beta}_1 + 2c) \sup_{ilm} \frac{\norm{D^3_{ilm} w}_{\infty;[s,T]}}{\beta^i} \\ 
& + \sqrt{T-s}\, (\norm{\beta}_1 + c) \sup_{il} \frac{\norm{D^2_{il} w}_{\infty;[s,T]}}{\beta^i}
+ Cc \sup_i \frac{\norm{D_i w}_{\infty;[s,T]}}{\beta^i}  \biggr) \beta^j,
\end{split}
\]
or, exploiting the Cauchy--Schwarz inequality, by
\[
\begin{split}
\sqrt{T-s}\,c_B \biggl( &3\sqrt{(T-s)c \norm{\beta}_1} \sup_{ilm} \frac{\norm{D^3_{ilm} w}_{\infty;[s,T]}}{\sqrt{\beta^i \beta^l}} \\
&+ 2\sqrt{T-s}\,c \sup_{il} \frac{\norm{D^2_{il} w}_{\infty;[s,T]}}{\beta^i} + Cc \sup_i \frac{\norm{D_i w}_{\infty;[s,T]}}{\beta^i}  \biggr) \sqrt{\beta^j \beta^k}.
\end{split}
\]
Finally, integrating by parts and using \eqref{tdli_estdkr} and \eqref{tdli_e3},
\[
\lim_{\epsilon\to0}\,\biggl| \int_s^T \int_{\R^N} D_{jkl}^3 F \rho_\epsilon \biggr| \leq Cc_F\sqrt{T-s}\, (\beta^j \wedge \sqrt{\beta^j \beta^k}\,).
\]
Collecting the estimates we have proved and also using \eqref{NST_estD1} and \eqref{NST_estD2} we can now obtain that
\begin{equation} \label{NST_estD3}
\sup_{ijk} \frac{\norm{D^3_{ijk} w}_{\infty;[s,T]}}{\beta^i \wedge \sqrt{\beta^i \beta^j}} \leq c_3(T-s), 
\end{equation}
for some $c_3$ increasing in $T-s$ and depending only on  $C_A''$, $C$ and the same parameters already listed above.

Plugging \eqref{NST_estD3} back into \eqref{NST_estD2} and the resulting estimate into \eqref{NST_estD1}, we conclude that there exists $\frk c(T-s)$ such that
\[
\sup_{i} \frac{\norm{D_{i} w}_{\infty;[s,T]}}{\beta^i} + \sup_{ij} \frac{\norm{D^2_{ij} w}_{\infty;[s,T]}}{\beta^i \wedge \sqrt{\beta^i \beta^j}} + \sup_{ijk} \frac{\norm{D^3_{ijk} w}_{\infty;[s,T]}}{\beta^i \wedge \sqrt{\beta^i \beta^j}} \leq \frk c(T-s),
\]
provided that $T-s$ is small enough, depending on $c, C_A'', c_B, c_F, c_G, \norm{\beta}_1, C$. It is now standard to iterate these estimates a finite number of times and get the desired ones on the whole interval $[0,T]$.
\end{proof}

Using the bounds we have obtained, it is also possible to deduce useful H\"older estimates in time of the space derivatives up to the second-order, as we state here below.

\begin{cor} \label{NST_cord2lipt}
Under the assumptions of \Cref{tdli_proppret}, one also has $w \in C^{\frac12}([0,T];C^2(\R^N))$, with
\begin{gather*}
\norm{w(s,\var) - w(t,\var)}_\infty \leq \tilde K'|t-s|^{\frac12},
\\
\norm{D^2_{j}w(s,\var) - D^2_{j}w(t,\var)}_\infty + \norm{D^2_{jk}w(s,\var) - D^2_{jk}w(t,\var)}_\infty  \leq \tilde K'|t-s|^{\frac12} \sqrt{\beta^j},
\end{gather*}
for all $t,s \in [0,T]$, $j,k \in \DSet N$; the constant $\tilde K'$ depends only on $\sup_j \norm{A^{jj}}_\infty$, $\sup_j \norm{B^j}_\infty$, $\norm{\beta}_{1}$, $\norm{F}_\infty$, $\tilde K$ and the parameters thereof. Furthermore, one can choose $\tilde K'=C\tilde K$ with $C$ linearly depending on $c_B$.
\end{cor}

\begin{proof}
The differentiability in time of $Dw$ comes directly from the equation. We prove now only the H\"older estimate for the second-order derivatives; the other ones are obtained in an analogous manner.
Testing over $(s,t)$ the equation of $D^2_{jk} w$ by $\rho_\epsilon$ (solving \eqref{tdli_kfp}, with arbitrary $y \in \R^N$) and letting $\epsilon \to 0$, the analogue of \eqref{NST_estD2torec} one obtains, after using all estimates on the data and $w$, is
\[
\biggl| D^2_{jk} w(s,y) - \int_{\R^N} D^2_{jk} w(t,\var)\,\di\rho(t) \biggr| \lesssim (t-s)(\beta^j \wedge \sqrt{\beta^j\beta^k}),
\]
with implied constants depending only on $\tilde K$ and the parameters thereof. Therefore,
\begin{equation} \label{est_D2hol}
 \begin{split}
&\bigl| D^2_{jk} w(s,y) - D^2_{jk} w(t,y) \bigr| 
\\
&\qquad\quad \lesssim (t-s)(\beta^j \wedge \sqrt{\beta^j\beta^k}) + \biggl| \int_{\R^N} D^2_{jk} w(t,\var)\,\di\rho(t) - D^2_{jk} w(t,y)\biggr| \\
&\qquad\quad \lesssim \biggl( (t-s) + \sum_{1 \leq l \leq N} \sqrt{\beta^l} \, \int_{\R^N} |\var^l - \, y^l|\, \di\rho(t) \biggr) \sqrt{\beta^j}.
\end{split}
\end{equation}
As $A$ and $B$ are bounded, we can test \eqref{tdli_kfp} by $|\var^l - y^l|^2$ to get
\[
\int_{\R^N} |\var^l - \, y^l|^2\, \di\rho(t) = \int_s^t \int_{\R^N} A^{ll} \,\di\rho - 2 \int_s^t \int_{\R^N} B^\ell (\var^l - \, y^l)\,\di\rho,
\]
so that Young's inequality and Gronwall's lemma yield
\[
\int_{\R^N} |\var^l - \, y^l|^2\, \di\rho(t) \leq (t-s)(\norm{A^{ll}}_\infty + \norm{B^l}_\infty^2)e^T.
\]
Using H\"{o}lder's inequality in \eqref{est_D2hol}, we obtain the desired estimate.%conclude that
%\[
%\bigl| D^2_{jk} w(s,y) - D^2_{jk} w(t,y) \bigr| \lesssim |t-s|^{\frac12}\sqrt{\beta^j},
%\]
%with implied constant depending also on $\sup_j \norm{A^{jj}}_\infty$, $\sup_j \norm{B^j}_\infty$ and $\norm{\beta}_{1}$.
\end{proof}

\section{Existence and uniqueness for the Nash system} \label{NST_secexu}

We first prove that, under suitable bounds on the Hamiltonian, the infinite-dimensional Nash system can have at most one classical solution, according to \Cref{NST_csoldef}.

\begin{thm} \label{NST_thmex}
Let $A \colon \R^\omega_T \to \ell^\infty(\N^2)$ be uniformly positive on $\ell^2(\N)$, with $D_i A^{hk}$ and $D^2_{ij}A^{hk}$ uniformly bounded for all $i,j,h,k \in \N$, and $D_k A^{ij} = 0$ if $i \neq k \neq j$. Let $H \colon \R^\omega_T \times \R^\omega \to \R$ be such that $\de_{p^i} H^i(t,\var)$ is twice differentiable in $\R^\omega \times \R^\omega$,
with
\[
\sup_{ik} \sum_j \bigl( |\de_{p^j} H^i| + |D_j \de_{p^j} H^i| + |\de^2_{p^jp^k} H^i| + |D_j \de^2_{p^jp^k} H^i| + |\de^3_{p^jp^kp^i} H^i| \bigr) < \infty
\] 
uniformly in $\R^\omega_T \times \Omega$ with $\Omega \subset \R^\omega$ $\infty$-bounded.\footnote{Recall that this means that $\sup_\Omega |\var|_\infty \defeq \sup_{p\in\Omega} \sup_i |p^i| < \infty$.}
Then there exists at most one classical solution to the Nash system~\eqref{tdli_ns}.\end{thm}

\begin{proof}
Let $u$ and $v$ be two classical solutions as in the statement. Let $w=u-v$, so that it solves
\[\begin{split}
&-\de_t w^i - \sum_{jk} A^{jk} D^2_{jk} w^i + \Bigl( \int_0^1 \de_p H^i(\call D(su+(1-s)v))\,\di s \Bigr) \cdot \call D w \\ 
&+ \sum_{j\neq i} \de_{p^j} H^j(\call Du) D_j w^i
+ \sum_{j\neq i} \Bigl( \int_0^1 \de^2_{pp^j} H^j(\call D(su+(1-s)v))\,\di s \Bigr) \cdot \call Dw \, D_jv^i = 0
\end{split}
\]
for each $i \in \N$, with terminal condition $w|_{t=T} = 0$.
Let now $\rho^i$ solve \eqref{tdli_kfp} where
\[
B^j = \begin{dcases}
\de_{p^j}H^j(\call Du) & \text{if $j \neq i$} \\
\int_0^1 \de_{p^i}H^i(\call D(su+(1-s)v))\,\di s & \text{if $j=i$};
\end{dcases}
\]
for the sake of simplicity we will do the computation assuming that $\rho(s) = \delta_y$, nevertheless what follows is to be understood in the limit as $\epsilon \to 0$.
Testing the equation for $w^i$ by $\rho^i$ over $(s,T)$ we get, with the notation $x_{-N} \defeq (x^j)_{j \geq N}$,
\[ \begin{split}
&w^i(s,y,x_{-N}) \\
&= \sum_{j \neq i} \int_s^T \int_{\R^N} \Big(  \Bigl( \int_0^1 \de_{p^j} H^i(\call D(s'u+(1-s')v))\,\di s' \Bigr) D_j w^j \Big)\Big|_{(t,(z,x_{-N}))} \rho^i(t,z)\,\di z\di t \\
&\quad +  \sum_{j \neq i} \int_s^T \int_{\R^N} \Big( \sum_k \Bigl( \int_0^1 \de^2_{p^kp^j} H^j(\ast)\,\di s' \Bigr) D_kw^k D_jv^i \Big)\Big|_{(t,(z,x_{-N}))} \rho^i(t,z)\,\di z\di t \\
&\quad + \sum_{j\geq N} \int_s^T \int_{\R^N} \Big( 2 \sum_k A^{jk} D^2_{jk}w^i + \de_{p^j} H^j(\call D u) D_j w^i \Big)\!\Big|_{(t,(z,x_{-N}))} \rho^i(t,z)\,\di z\di t,
\end{split}
\]
where for the sake of brevity we used $\ast$ as a placeholder for $\call D(s'u+(1-s')v)$.
As long as $j<N$, integrating by parts one finds
\[ \begin{split}
&- \int_{\R^N} \!\Big(  \Bigl( \int_0^1 \de_{p^j} H^i(\ast)\,\di s' \Bigr) D_j w^j \Big)\Big|_{(t,(z,x_{-N}))} \rho^i(t,z)\,\di z\di t \\
&= \int_{\R^N} \!\Big( \Bigl( \int_0^1 \de_{p^j} H^i(\ast)\,\di s' \Bigr) w^j \Big)\Big|_{(t,(z,x_{-N}))} D_j \rho^i(t,z)\,\di z\di t \\
&\quad + \int_{\R^N} \!\Big(  \Bigl( \int_0^1 D_j \de_{p^j} H^i(\ast)\,\di s' \Bigr) w^j \Big)\Big|_{(t,(z,x_{-N}))} \rho^i(t,z)\,\di z\di t \\
&\quad + \int_{\R^N} \!\Big( \sum_\ell \Bigl( \int_0^1 \de^2_{p^\ell p^j} H^i(\ast)\,\di s' \Bigr) D^2_{j\ell}(s u^\ell + (1-s)v^\ell) \,w^j \Big)\Big|_{(t,(z,x_{-N}))} \rho^i(t,z)\,\di z\di t;
\end{split}
\]
analogously,
\[ \begin{split}
&- \int_{\R^N} \Big( \Bigl( \int_0^1 \de^2_{p^kp^j} H^j(\ast)\,\di s' \Bigr) D_kw^k D_jv^i \Big)\Big|_{(t,(z,x_{-N}))} \rho^i\,\di z\di t \\
&= \int_{\R^N} \Bigl(\Bigl( \int_0^1 \de^2_{p^kp^j} H^j(\ast)\,\di s' \Bigr) w^k D^2_{jk} v^i \Bigr)\Bigr|_{(t,(z,x_{-N}))} \rho^i\,\di z \\
&\quad+ \int_{\R^N} \Bigl( \Bigl( \int_0^1 \de^2_{p^kp^j} H^j(\ast)\,\di s' \Bigr) w^k D_{j} v^i \Bigr)\Bigr|_{(t,(z,x_{-N}))} D_k\rho^i\,\di z \\
&\quad+ \int_{\R^N} \Bigl( \Bigl( \int_0^1 D_k \de^2_{p^kp^j} H^j(\ast)\,\di s' \Bigr) w^k D_{j} v^i \Bigr)\Bigr|_{(t,(z,x_{-N}))} \rho^i\,\di z \\
&\quad + \int_{\R^N} \!\!\Big( \sum_{\ell} \Bigl( \int_0^1 \de^3_{p^\ell p^kp^j} H^j(\ast)\,\di s' \Bigr) D^2_{k\ell} (s u^\ell + (1-s)v^\ell) \,w^k D_j v^i \Big)\Big|_{(t,(z,x_{-N}))} \rho^i\,\di z\,,
\end{split}
\]
where $\rho^i = \rho^i(t,z)$.
Therefore, using \Cref{tdli_lem:kfp} and the properties required to classical solutions, we deduce that for $T-s < T^*$
\[
\sup_i\,\norm{w^i}_{\infty;[s,T]} \lesssim\, \sqrt{T-s}\, \Bigl( \sup_i\, \norm{w^i}_{\infty;[s,T]} + \varepsilon_N \Bigr)
\]
with $\varepsilon_N \to 0$ as $N \to \infty$ and implied constant independent of $N$. It follows that if $s$ is sufficiently close to $T$ one has $w^i = 0$ for all $i$ (that is, $u=v$) on $[s,T] \times \R^\omega$.
Finally define $\hat s \defeq \min\{ s \in [0,T] :\ u=v \text{ on } [s,T] \times \R^\omega \}$; if $\hat s > 0$ then by the argument above with $\hat s$ in lieu of $T$ we get a contradiction, thus $\hat s = 0$ and the proof is complete.
\end{proof}

To prove existence, we solve a version of the Nash system~\eqref{tdli_ns} reduced onto $\R^N_T \defeq [0,T] \times \R^N$ and then pass to the limit as $N \to \infty$ to obtain a solution to the infinite-dimensional system, thanks to the stability provided by our estimates of \Cref{tdli_proppret}. This is going to be done by observing that letting $B_i(t,x) \defeq \hat B_i(t,x,\call D(t,x))$, with
\begin{equation} \label{NST_defBU}
\hat B^{j}_i(t,x,p) \defeq \begin{dcases}
\de_{p^j} H^j(t,x,p) & \text{if}\ j \neq i \\
\int_0^1 \de_{p^i} H^i(t,x,p^{-i}, sp^i) \,\di s & \text{if} \ j = i\,,
\end{dcases}
\end{equation}
and
\begin{equation} \label{NST_defFU}
F^i(t,x) \defeq H^i(t,x,\call Du^{-i}(t,x),0),
\end{equation}
the equations in \eqref{tdli_ns} take the form $-\de_t u^i - \tr(AD^2u^i) + \pair{B_i}{D u^i} = F^i$,
in such a way that they can be regarded as a system of linear transport-diffusion equations in the context of a fixed point argument.

When performing the reduction from $\R^\omega_T$ to $\R^N_T$, we will identify $\R^N \subset \R^\omega$ via $(x^0,\dots,x^{N-1}) \mapsto (x^0,\dots,x^{N-1},0,\dots)$. Given $V \colon \R^\omega \to \R$ we will consider the projection $V_N(y) \defeq V(y)$ for all $y \in \R^N$, while for $A \colon \R^\omega \to \ell^\infty(\N^2)$ the notation $A_N$ will mean that we have also restricted the infinite-dimensional matrix to $\R^N$; that is, $A_N(y) \defeq A(y)|_{\DSet N} = (A(y)^{ij})_{i,j \in \DSet N}$ for all $y \in \R^N$. Finally, the spaces defined in \Cref{NST_secholsp} are readily adapted to functions on $\R^N$, simply by replacing $\R^\omega$ with $\R^N$.

In the upcoming statement, which is our main theorem, we will make use of the notations $\R^\omega_x$ and $\R^\omega_p$ presented in the opening section; $\R^\omega_T$ will always mean $[0,T] \times \R^\omega_x$.

\begin{thm} \label{tdli_thmone}
Let $\beta \in \ell^{\frac12}(\Z,\R_+)$ be even and such that $\beta \star \beta \leq c\beta$; for any $i \in \N$, define $\beta_i \defeq \beta^{i-\var}$. Assume that the following hypotheses are fulfilled, uniformly in the parameters $i,j,k \in \N$:
\begin{itemize}[leftmargin=2em]
\item $A \colon \R^\omega_T \to \ell^\infty(\N^2)$ is uniformly positive on $\ell^2(\N)$, with
\[
A^{ij} = A^{ji} \in C^{\frac\gamma2,\gamma}(\R^\omega_T) \quad \text{for some $\gamma \in (0,1)$}\,,
\]
$D_k A^{ij} = 0$ unless $i=j=k$, and
\[
D_k A^{kk} \in C^{\frac\gamma2,\gamma}(\R^\omega_T) \cap C^0([0,T];C^{2}(\R^\omega))\,;
\]
\item $G^i \in C^{3+\gamma-}_{\beta_i}(\R^\omega)$;
\item for any $\infty$-bounded $\Omega \subset \R^\omega_p$,
\[
\de_{p^i} H^i \in C^0([0,T];C^2_{\beta_i}(\R^\omega_x \times \Omega)) \cap C^0(\R^\omega_T;C^{3-}_{\beta_i}(\Omega))\,;
\]
also,
\[
D_j \de_{p^i} H^i,\, \de_{p^j} \de_{p^i} H^i \in C^0(\Omega;C^{\frac\gamma2,\gamma}(\R^\omega_T))\,.
\]
\end{itemize}

Then there exists $T^*>0$ such that if $T < T^*$ the Nash system~\eqref{tdli_ns} has a unique classical solution $u$ (with $\beta_i$ defined as above), which also belongs to $\ell^\infty(\N;C^0([0,T];C^{2,1}_{\beta_i}(\R^\omega)) \cap C^{\frac12}([0,T];C^{2-}_{\sqrt{\beta_i}}(\R^\omega)))$.

%Furthermore, if $\sup_i \norm{\de_{p^i} H^i}_{C^0([0,T];C^2_{\beta_i}(\R^\omega \times \Omega)) \cap C^0(\R^\omega_T;C^{3-}_{\beta_i}(\Omega))} \leq c_H(\,\sup_\Omega|\var|_\infty)$, with $c_H$ of at most polynomial growth, then $T^*=+\infty$.
\end{thm}

\begin{proof}
We only need to prove existence, since uniqueness will then follow because the hypotheses of \Cref{NST_thmex} are satisfied.

Let $N \in \N$ and let $\call Y$ be the closed subset of
\[
\bigl(C^0([0,T];C^{2,1}_{\beta_i}(\R^N)) \cap C^{\frac12}([0,T];C^{2-}_{\sqrt{\beta_i}}(\R^N))\bigr)^N
\]
of all $u$ such that
\begin{equation} \label{NST_boundY1'}
\sup_{i} \, \norm{u^i}_{C^0([0,T];C^{3-}_{\beta_i}(\R^N))} \leq R \quad\text{and}\quad \sup_{i} \, \norm{u^i}_{C^{\frac12}([0,T];C^{2-}_{\sqrt{\beta_i}}(\R^N))} \leq R',
\end{equation}
for some $R,R'>0$ to be determined.
We will denote by $\trinorm u$ the norm on $(C^0([0,T];C^{2,1}_{\beta_i}(\R^N)) \cap C^{\frac12}([0,T];C^{2-}_{\sqrt{\beta_i}}(\R^N)))^N$, defined as the sum of the two norms above. Given $u \in \call Y$, let $\call S$ be the map that associates to $u$ the solution $w$ to
\[
\begin{dcases}
-\de_t w^i - \sum_{j,k \in \DSet N} A^{jk}_N D^2_{jk} w^i + \sum_{j \in \DSet N} B^j_{i,N} D_jw^i = F^i_N & \text{in $\R^N_T$} \\
w^i|_{t=T} = G^i_N
\end{dcases} \quad i \in \DSet N\,,
\]
with $B$ and $F$ being defined as in \eqref{NST_defBU} and \eqref{NST_defFU}; here we have used the notation with the subscript $N$ that we have previously introduced. By \cite[Theorem~8.12.1]{kryHS}, $w$ and $Dw$ are of class $C^{1+\frac{\gamma'}2, 2+\gamma'}$, with $\gamma' \defeq \gamma \wedge \frac12$. Also, it is not difficult to see that there exists a constant $\kappa_R$ which depends only on $\beta$, $c_H$ and $R$ such that
\[
\sup_{ih} \bigl( \norm{B^{h}_i}_{C^0([0,T];C^2_{\beta_h}(\R^N))} + \norm{F^i}_{C^0([0,T];C^2_{\beta_i}(\R^N))}\bigr) \leq \kappa_R\,,
\]
so, by \Cref{tdli_proppret} and \Cref{NST_cord2lipt} one can choose $R>\sup_i \norm{G^i}_{3+\gamma-;\beta_i}$, $T$ sufficiently small and $R'$ large enough in such a way that, for all $N \in \N$, $\call S$ is well-defined with values in $\call Y$.

Consider now $u,v \in \call Y$ and denote by $B_{i,N}[u]$ and $B_{i,N}[v]$ the corresponding drifts;\footnote{That is, $B_{i,N}[u](t,x)$ is a short notation for $\hat B_{i,N}(t,x,\call Du(t,x))$, according to \eqref{NST_defBU}, and analogously for $v$.} then, letting $\bar w = \call S(u) - \call S(v)$ and $\tilde F^i_N = \pair{ B_{i,N}[v] -  B_{i,N}[u]}{D \call S(v)^i}$ we have
\[ \begin{dcases}
-\de_t \bar w^i - \sum_{jk} A^{jk}_N D^2_{jk} \bar w^i + \pair{ B_{i,N}[u]}{D\bar w^i} = \tilde F^i_N \\
\bar w|_{t=T} = 0.
\end{dcases}
\]
Note that
\begin{equation} \label{NST_formtF}
\tilde F^i = \sum_j \Bigl( \int_0^1 \de_p \hat B_{i,N}^j[su + (1-s)v] \,\di s \Bigr) \cdot \call D (u-v) \, D_j \call S(v)^i,
\end{equation}
hence, letting $c_{\tilde F} \defeq \sup_i \norm{\tilde F^i}_{C^0([0,T];C^2_{\beta_i}(\R^N))}$, we have $c_{\tilde F} \leq \kappa_R' \trinorm{u-v}$, with $\kappa'_R$ independent of $N$. In addition, by \Cref{tdli_proppret},
\[
\sup_i \norm{\bar w^i}_{C^0([0,T];C^{3-}_{\beta_i}(\R^N))} \leq c_{\tilde F} \tilde K\,.
\]
Then, also using \Cref{NST_cord2lipt}, we deduce that there exists $\bar T > 0$, independent of $N$, such that if $T \leq \bar T$ then $\trinorm{\call S(u) - \call S(v)} \leq \tfrac12 \trinorm{u-v}$; hence the contraction theorem yields the existence of a unique fixed point for $\call S$ in $\call Y$, which is by its definition the solution $u_N$ to the Nash system on $\R^N_T$ with data $A_N$, $H_N$ and $G_N$.

%At this point we claim that such a solution can be extended on an arbitrary (but fixed) horizon $[0,T]$, provided that $c_H$ has at most polynomial growth. To see this, for simplicity consider the equation in forward form (that is, after applying the change of variables $t \mapsto T-t$), then let $\epsilon > 0$ small and argue as above with $R = c_G + \epsilon$. Note that, in order to have that $\call S$ is a contraction on $\call Y$, one can choose $T_0 = \epsilon/\kappa_{c_G+\epsilon}''$, for some increasing and continuous function $R \mapsto \kappa''_R$ independent of $R'$ and with at most polynomial growth. Iterating this procedure, one can extend the solution up onto $[0, T_\epsilon)$ with $T_\epsilon = \epsilon \sum_{i\in\N} (\kappa_{c_G+i\epsilon}'')^{-1}$; therefore, since $T_\epsilon \to \infty$ as $\epsilon \to 0$, given $T>0$ it suffices to choose $\epsilon$ such that $T_\epsilon > T$.

For any $\delta \in (0,\frac12)$, by the Ascoli--Arzelà theorem and a standard diagonal argument, up to a subsequence which is common to all $i \in \N$, each $(u_N^i)_{N\geq i}$ compactly converges in $C^0([0,T];C^{2+\delta}_{\beta_i}(\R^\omega)) \cap C^\delta([0,T];C^{2-}_{\sqrt{\beta_i}}(\R^\omega))$ to a function $u^i$ enjoying the bounds~\eqref{NST_boundY1'}.

Since $\beta \in \ell^{\frac12}(\Z)$, by the dominated convergence theorem
\begin{gather*}
\sum_{j,k< N} A^{jk}_N D^2_{jk} u_N^i \,\longrightarrow\, \sum_{j,k \in \N} A^{jk} D^2_{jk} u^i \,, \\
\sum_{\substack{j < N \\ j\neq i}} \de_{p^i} H^j_N(\call D u_N) D_j u_N^i \,\longrightarrow\, \sum_{\substack{j \in \N\\ j\neq i}} \de_{p^i} H^j(\call D u) D_j u^i
\end{gather*}
compactly in $C^0(\R^\omega_T)$ as $N \to \infty$. As also $H^i_N(\call Du_N) \to H^i(\call Du)$ compactly in $C^0(\R^\omega_T)$, by the equation we have that so $\de_t u_N^i$ compactly converges in $C^0(\R^\omega_T)$ as well. In particular, for any $x \in \R^\omega_T$, $\de_t u^i_N(\var,x)$ uniformly converges to some $v^i(\var,x)$ which is thus $\de_t u^i(\var,x)$ by the fundamental theorem of calculus. We conclude that $u$ solves \eqref{tdli_ns} pointwise. Furthermore, by the decay estimates on the derivatives we see that the series appearing in the equations converge in $C^0(\R^\omega_T)$. To complete the proof, we need to show that such a convergence is uniform in $i$.

Note that
\[
\sum_{0\leq j,k \leq N} \norm{D^2_{jk} u^i}_\infty \lesssim \biggl(\, \sum_{0\leq j \leq N} \gamma^{j-i} \biggr)^2,
\]
where we have set $\gamma^i \defeq \sqrt{\beta^i}$ and the implied constant is independent of $i \in \N$. The latter sum equals
\[ \begin{dcases}
\sum_{i-N \leq k \leq i} \gamma^k & \text{if $i \geq N$}
\\
\gamma^0 +\ 2 \sum_{1 \leq k \leq i \wedge (N-i)} \gamma^k \ + \sum_{i \wedge (N-i) < k \leq i \vee (N-i)} \gamma^k & \text{if $i < N$}.
\end{dcases}
\]
Since $(\gamma^k)_{k\geq 0}$ is summable, $\sum_{i-N \leq k \leq i} \gamma^k \to 0$ as $i \to \infty$, thus there exists $\bar \imath = \bar \imath(N) \geq N$ such that
$\sum_{i-N \leq k \leq i} \gamma^k \leq \sum_{0 \leq k \leq \bar \imath} \gamma^k$ for all $i \geq N$;
then one easily sees that
\[
\sum_{0\leq j,k \leq N} \norm{D^2_{jk} u^i}_\infty \lesssim \biggl(\, \sum_{0 \leq k \leq \bar\imath} \gamma^k \biggr)^2 \quad \forall\,i \in \N.
\]
Similarly, there exists $\tilde\imath = \tilde\imath(N) \geq N$ such that
\[
\sum_{0\leq j \leq N} \norm{D_j u^j D_j u^i}_\infty \,\lesssim\, \beta^0 \sum_{0\leq j \leq N} \beta^{j-i} \,\lesssim\, \sum_{0\leq k \leq \tilde\imath(N)} \beta^k \quad \forall\, i \in \N.
\]
Then the series appearing in the Nash system converge in $C^0(\R^\omega_T)$ uniformly in $i \in \N$.
\end{proof}

\begin{rmk}
Given the local nature (in time) of the result, and the consequent possibility to perform a fixed point argument to achieve it, we do not need to require any global boundedness of second and higher order $p$-derivatives of the Hamiltonians (as it would be expected, e.g., to prove well-posedness of the master equation in MFGs on arbitrary horizons). In fact, only some \emph{local} boundedness is needed, and this is in line with some analogous results on \emph{short-time} existence for the master equation in MFGs, such as \cite{CCPJEMS}.
\end{rmk}

%\bibliography{IDNSbib}

\begin{thebibliography}{99}
\bibitem{Ahuja16} {\sc S. Ahuja}, {\em Wellposedness of Mean Field Games with Common Noise under a Weak Monotonicity Condition}, SIAM J. Control Optim., 54, 1 (2016), pp.~30--48.
\bibitem{Carmona2} {\sc A. Aurell, R. Carmona, and M. Lauri\`{e}re}, {\em Stochastic graphon games: II. The linear-quadratic case}, Appl. Math. Optim., 85, 3 (2022), Article No. 39.
\bibitem{BM24} {\sc M. Bansil and A. R. M\'{e}sz\'{a}ros}, {\em Hidden monotonicity and canonical transformations for mean field games and master equations}, Forum Math. Sigma, 13 (2025), e182.
\bibitem{BWZ} {\sc E. Bayraktar, R. Wu, and X. Zhang}, {\em Propagation of chaos of forward-backward stochastic differential equations with graphon interactions}, Appl. Math. Optim., 88, 1 (2023), Article No. 25.
\bibitem{BLL_HJBID} {\sc C. Bertucci, J.-M. Lasry, and P.-L. Lions}, {\em A singular infinite-dimensional Hamilton--Jacobi--Bellman equation arising from a storage problem}, Math. Models Methods Appl. Sci., 35, 3 (2025), pp.~703--731. 
\bibitem{BD_HJBID} {\sc C. Bianca and C. Dogbe}, {\em Regularization and propagation in a Hamilton--Jacobi--Bellman-type equation in infinite-dimensional Hilbert space}, Symmetry, 16, 8 (2024).
\bibitem{CainesIEEE22} {\sc P. E. Caines}, {\em Embedded vertexon-graphons and embedded GMFG systems}, in 2022 IEEE 61st Conference on Decision and Control, pp.~5550--5557.
\bibitem{CainesHuang} {\sc P. E. Caines and M. Huang}, {\em Graphon mean field games and their equations}, SIAM J. Control Optim., 59, 6 (2021), pp.~4373--4399.
\bibitem{CainesHuACC24} {\sc P. E. Caines and M. Huang}, {\em Mean field games on dense and sparse networks: The graphexon MFG equations}, in 2024 American Control Conference, pp.~4230--4235.
\bibitem{CainesHuIEEE24} {\sc P. E. Caines and M. Huang}, {\em Sparse network mean field games: Ring structures and related topologies}, in 2024 IEEE 63rd Conference on Decision and Control, pp.~2584--2590.
\bibitem{HJBHHS} {\sc A. Calvia, G. Cappa, F. Gozzi, and E. Priola}, {\em HJB equations and stochastic control on half-spaces of Hilbert spaces}, J. Optim. Theory Appl., 198 (2023), pp.~710--744.
\bibitem{CDLL} {\sc P. Cardaliaguet, F. Delarue, J.-M. Lasry, and P.-L. Lions}, {\em The master equation and the convergence problem in mean field games}, Ann. of Math. Stud., 201, Princeton University Press, Princeton, NJ, 2019.
\bibitem{CPnotes} {\sc P. Cardaliaguet and A. Porretta}, {\em An introduction to mean field game theory}, Lecture Notes in Math., 2281, Springer, Cham, 2020.
\bibitem{CCPJEMS} {\sc P. Cardaliaguet, M. Cirant, and A. Porretta}, {\em Splitting methods and short time existence for the master equations in mean field games}, J. Eur. Math. Soc., 25 (2023), pp.~1823--1918.
\bibitem{cardelar} {\sc R. Carmona and F. Delarue}, {\em Probabilistic theory of mean field games with applications. I: Mean field FBSDEs, control, and games}, Probab. Theory Stoch. Model., 83, Springer, Cham, 2018. 
\bibitem{CiPo} {\sc M. Cirant and A. Porretta}, {\em Long time behavior and turnpike solutions in mildly non-monotone mean field games}, ESAIM Control Optim. Calc. Var., 27 (2021), Article No. 86.
\bibitem{CR_ENS} {\sc M. Cirant and D. F. Redaelli}, {\em A priori estimates and large population limits for some nonsymmetric Nash systems with semimonotonicity}, Comm. Pure Appl. Math. 79 (2026), pp.~3--88.
\bibitem{CR24} {\sc M. Cirant and D. F. Redaelli}, {\em Some remarks on Linear-Quadratic closed-loop games with many players}, Dyn. Games Appl., 15 (2025), pp.~558--591.
\bibitem{DaPrato_Zabczyk_2002} {\sc G. Da Prato and J. Zabczyk}, {\em Second Order Partial Differential Equations in Hilbert Spaces}, London Math. Soc. Lecture Note Ser., Cambridge University Press, 2002.
\bibitem{fried} {\sc A. Friedman}, {\em Partial Differential Equations of Parabolic Type}, Prentice-Hall, 1964.
\bibitem{Friedman72} {\sc A. Friedman}, {\em Stochastic differential games}, J. Differential Equations, 11 (1972), pp.~79--108.
\bibitem{gangbo} {\sc W. Gangbo, A. R. M\'{e}sz\'{a}ros, C. Mou, and J. Zhang}, {\em Mean field games master equations with nonseparable Hamiltonians and displacement monotonicity}, Ann. Probab. 50, 6 (2022), pp.~2178--2217.
\bibitem{GraMe} {\sc P. J. Graber and A. R. M\'{e}sz\'{a}ros}, {\em On monotonicity conditions for mean field games}, J. Funct. Anal., 285, 9 (2023), Article No. 110095.
\bibitem{hcm} {\sc M. Huang, R. P. Malham\'e, and P. E. Caines}, {\em Large population stochastic dynamic games: closed-loop McKean-Vlasov systems and the Nash certainty equivalence principle}, Commun. Inf. Syst., 6, 3 (2006), pp.~221--251.
\bibitem{keltop} {\sc J. L. Kelley}, {\em General topology}, Van Nostrand, Toronto, Ont., 1955.
\bibitem{kryHS} {\sc N. V. Krylov}, {\em Lectures on Elliptic and Parabolic Equations in H\"{o}lder Spaces}, Grad. Stud. Math., 12, American Mathematical Society, 1996.
\bibitem{LackRamWu} {\sc D. Lacker, K. Ramanan, and R. Wu}, {\em Local weak convergence for sparse networks of interacting processes}, Ann. Appl. Probab., 33, 2 (2023), pp.~843--888.
\bibitem{LackerSoretLabel} {\sc D. Lacker and A. Soret}, {\em A label-state formulation of stochastic graphon games and approximate equilibria on large networks}, Math. Oper. Res., 48, 4 (2023), pp.~1987--2018.
\bibitem{LL07} {\sc J.-M. Lasry and P.-L. Lions}, {\em Mean field games}, Jpn., J., Math., 2, 1 (2007), pp.~229--260.
\bibitem{LionsSeminar} {\sc P.-L. Lions}, {\em Cours du Coll\`{e}ge de France} (2008), available at \url{https://www.college-de-france.fr/fr/agenda/cours/jeux-champ-moyen-suite-2}.
\bibitem{Masiero16} {\sc F. Masiero}, {\em HJB equations in infinite dimensions under weak regularizing properties}, J. Evol. Equ., 16 (2016), pp.~789--824.
\bibitem{Masiero14} {\sc F. Masiero and A. Richou}, {\em HJB equations in infinite dimensions with locally Lipschitz Hamiltonian and unbounded terminal condition}, J. Differ. Equ., 257, 6 (2014), pp.~1989--2034.
\bibitem{MouZhang} {\sc C. Mou and J. Zhang}, {\em Mean field game master equations with anti-monotonicity conditions}, J. Eur. Math. Soc., 27, 11 (2025), pp.~4469--4499.
\bibitem{PHDT} {\sc D. F. Redaelli}, {\em On nonlinear systems of PDEs arising in the theory of differential games}, Doctoral Thesis, University of Padua, 2024.
\end{thebibliography}
%\bibliographystyle{abbrv}

\end{document}